\documentclass[11pt,twoside]{article}
\usepackage{amsmath,amssymb,amsthm,bm,mathrsfs,amscd}
\usepackage{graphicx}
\usepackage{epstopdf}
\usepackage{caption}
\usepackage{tabu}
\usepackage{multirow}
\usepackage{tikz-cd}
\usepackage{subfigure}
\usepackage{booktabs}
\usepackage{diagbox}
\usepackage{algorithm}
\usepackage{algpseudocode}
\usepackage{listings}
\usepackage{color,xcolor}

\allowdisplaybreaks[4]

\lstdefinelanguage{Maple}{
  morekeywords={and,assuming,break,by,catch,description,do,done,elif,else,
                end,error,export,finally,for,from,global,if,implies,in,
                intersect,local,minus,mod,module,next,not,od,option,options,
                or,proc,quit,read,return,save,stop,subset,then,to,try,
                union,use,uses,while,xor}, 
  morecomment=[l]{\#}, 
  morecomment=[s]{\{}{\}}, 
  morestring=[b]" 
}

\makeatletter
\@addtoreset{equation}{section}
\makeatother

\newtheorem{exam}{Example}[section]

\newtheorem{thm}{Theorem}[section]
\newtheorem{rmk}{Remark}[section]

\newcounter{saveeqn}%

\makeatletter
\evensidemargin 
\oddsidemargin
\def\thanks#1{
	\protected@xdef\@thanks{
		\@thanks\protect\footnotetext{#1}
	}
}
\makeatother

\title{
\Large\bf{Normal form method of center-focus problem in piecewise-smooth systems and algorithm design}
}

\author{Xingwu Chen$^1$, ~~Jiahao Li$^{2,*}$,
\footnote{*~Corresponding author. }
Tao Li$^3$
\footnote{Email address: scuxchen@scu.edu.cn (Xingwu Chen), ljhmath@swu.edu.cn (Jiahao Li), litao@swufe.edu.cn (Tao Li). }
\\
{\small 1. School of Mathematics, Sichuan University, Chengdu, Sichuan 610064, P. R. China}\\
{\small 2. School of Mathematics and Statistics, Southwest University,}\\
  {\small  Chongqing 400715, P. R. China}\\
{\small 3. School of Mathematics, Southwestern University of Finance and Economics,}\\
  {\small Chengdu, Sichuan 611130, P. R. China}
}
\date{}

\begin{document}
\maketitle

\begin{abstract}
We investigate the normal form method in piecewise-smooth monodromic systems to distinguish centers from foci, to determine focus order and stability, and to supply algebraically equivalent computable quantities for degenerate Hopf bifurcation analyses.
For the normal form of piecewise-smooth monodromic systems given in previous publications in the sense of topological equivalence,
we prove the algebraic equivalence between the Lyapunov constants of the original piecewise-smooth system and the
coefficient series of its normal form.
With this algebraic equivalence, we prefer to compute the normal form coefficient series
to avoid cumbersome integrals of trigonometric functions in the computation of Lyapunov constants,
and find the relations of foci of the piecewise-smooth system and its subsystems:
among the orders and among the stabilities.
Moreover,
we design algorithms for computing the normal form coefficient series and combine them with the obtained stability relation
to construct stabilizing switching signals for state-dependent switched nonlinear systems, one basic problem in the control theory.
\end{abstract}

\baselineskip 15pt
\parskip 10pt
\thispagestyle{empty}
\setcounter{page}{1}

\section{Introduction}

\setcounter{equation}{0}
\setcounter{lm}{0}
\setcounter{thm}{0}
\setcounter{rmk}{0}
\setcounter{df}{0}
\setcounter{cor}{0}

A singular point is called monodromic if no orbits approach it in a definite direction as $t\to \pm\infty$ (see \cite{AN,ROM,ZZ}).
In the study of planar nonlinear differential equations, two fundamental problems arise concerning such a point.
One is the center-focus problem: determining whether a monodromic singular point is a center or a focus.
The other is the Hopf cyclicity
problem: determining the maximum number of limit cycles that can bifurcate from a monodromic singular point.
Both problems have a long history and have attracted the attention of many scholars (see, e.g., \cite{Bau,GIN,HU,PEA,ROM} and the references therein).

Recent decades have seen a surge of interest in the study of the center-focus and Hopf cyclicity problems for piecewise-smooth systems (see, e.g., \cite{BUZC,BUZ,CX1,CX2,CO,HAN,LIA,TIAN}),
which arise naturally in numerous practical applications such as mechanical systems with dry friction, switching control, electronic systems (see \cite{BDM,JMR,ZOU} and the references therein),
with particular emphasis on piecewise-smooth systems of the form
\begin{align}\label{ge}
		\left( \begin{array}{c}
		\dot{x}\\
		\dot{y}\\
	\end{array} \right) =\left\{
    \begin{aligned}
&		\left( \begin{array}{c}
		X^+(x,y)\\
		Y^+(x,y)\\
	\end{array} \right), ~~~~&&\mathrm{if}~y>0,\\
&		\left( \begin{array}{c}
		X^-(x,y)\\
		Y^-(x,y)\\
	\end{array} \right),   && \mathrm{if}~y<0,
	\end{aligned} \right.
\end{align}
where $X^{\pm}$, $Y^{\pm}$ are real analytic functions and the origin $O: (0,0)$ is a monodromic singular point.
The piecewise-smooth system~\eqref{ge} is split into two subsystems by the line $y=0$, which is called a {\it switching line}.
The subsystem defined on the half plane $y>0$ (resp. $y<0$) is called the {\it upper} (resp. {\it lower}) subsystem.
Usually, the orbits of system~\eqref{ge} are defined using Filippov's convention, as described in \cite{FIL}.
Compared with classical analytic systems, the challenges in studying the center-focus and Hopf cyclicity problems for the piecewise-smooth system~\eqref{ge} come from the following two aspects.

The first challenge is {\it that more types of monodromic singular points, generated by the switching structure of system~\eqref{ge}, need to be considered}.
Indeed, a focus or an invisible tangent point of the subsystems can compose three types of monodromic singular points, namely {\it focus-focus} (FF), {\it focus-parabolic} (FP) and {\it parabolic-parabolic} (PP) as shown in \cite[Figure 1]{LIJ} (here Figure~\ref{MSP} for convenience).
\begin{figure}[htp]
  \centering
  \includegraphics[width=15cm]{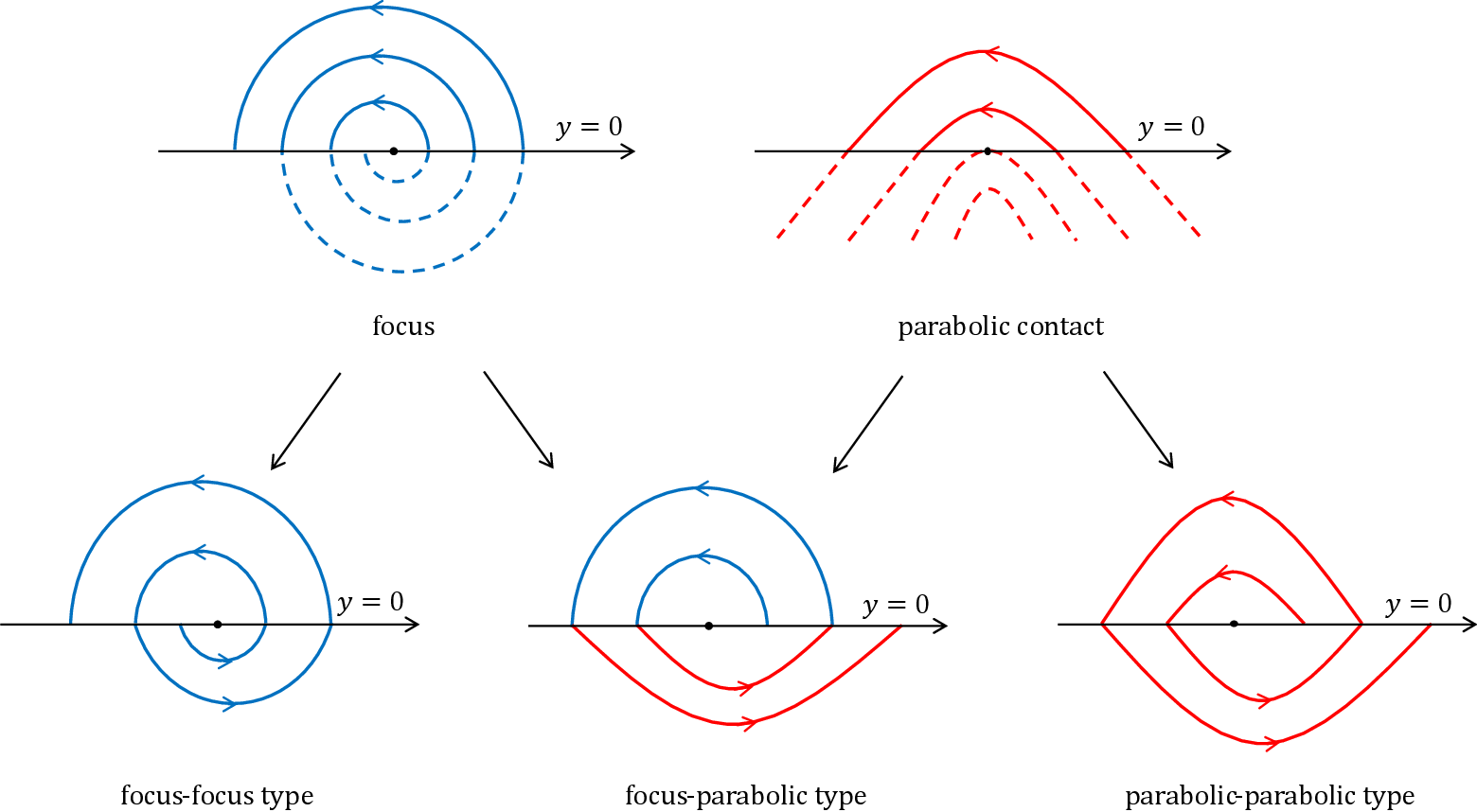}
  \caption{Three types of monodromic singular points of system~\eqref{ge}. }\label{MSP}
\end{figure}
Note that the parabolic-focus type can be reduced to FP type by the transformation $(x,y,t)\to (-x,-y,t)$.
We assume that the focus side subsystem has a linear focus $O$, i.e., its Jacobian matrix at $O$ has eigenvalues $\alpha\pm\beta i$ ($\beta\ne 0$),
and assume that the parabolic side subsystem has an invisible tangent point $O$ with multiplicity $2\ell-1$ ($\ell\in \mathbb{Z}^+$), i.e., it satisfies
\begin{equation*}
  Y(0,0)=\frac{\partial Y(0,0)}{\partial x}=\cdots=\frac{\partial^{2\ell-2}Y(0,0)}{\partial x^{2\ell-2}}=0 ~~{\rm and~~}
  X(0,0)\frac{\partial^{2\ell-1}Y(0,0)}{\partial x^{2\ell-1}}=\delta\ne0,
\end{equation*}
where $(X,Y,\ell,{\rm sgn}\delta)=(X^+,Y^+,\ell^+,-1)$ or $(X^-,Y^-,\ell^-,1)$.

The second challenge is {\it that the generic approach or computer algorithm for computing Lyapunov constants or their algebraically equivalent quantities
for the monodromic singular point of system~\eqref{ge} is still incomplete.}
Note that once Lyapunov constants or their algebraically equivalent quantities (the definition is given in section 2) are obtained, the following analysis based on these constants can be carried out using similar methods established for classical analytic systems (see \cite{ROM}).
Many significant contributions have been made to address this challenge by directly computing Lyapunov constants.
Using $(R,\theta,p,q)$-generalized polar coordinates, \cite{CO} derived formulas for the first three Lyapunov constants of FF, FP and PP types.
For the higher-order Lyapunov constants of FF type, \cite{GAS} established an algorithm based on a decomposition of certain one-forms in polar coordinates, and later \cite{TIAN} developed a new recursive procedure.
For PP type, a general formula for computing Lyapunov constants was given in \cite{NOV}.
To the best of our knowledge, no general algorithm exists for computing Lyapunov constants of FP type.

Computing Lyapunov constants involves cumbersome integrals of trigonometric functions, which are difficult to evaluate by hand or by computer.
For analytic systems, this difficulty can be avoided by the normal form method or the formal series method, which computes algebraically equivalent quantities of Lyapunov constants (see \cite{HANB,ROM}).
Based on those methods, many efficient computer algorithms have been developed (see, e.g., \cite{LLO,PEA,YU}).
Consequently, many scholars have attempted to extend the normal form method to the study of the piecewise-smooth system~\eqref{ge}.
For FF type, a second-order normal form and an arbitrary-order normal form are given in
\cite{CX} and \cite{EM2} respectively. But there is no clear relationship between such normal forms and Lyapunov constants except
the first two Lyapunov constants obtained in \cite{CX}.
For PP type, \cite{EM} gave an arbitrary-order normal form, but the Lyapunov constants remain unknown except for a special case in \cite{EM1}.
Recently, \cite{LIJ} provided new arbitrary-order normal forms for all three types, FF, FP, and PP, and established the algebraic equivalence between the Lyapunov constant series of these normal forms
and their coefficient series (see \cite[Theorem 1.2]{LIJ}).

However, the previous works only proved that the transformation from system~\eqref{ge} to its normal form preserves the local topological structure, without establishing the algebraic equivalence of the Lyapunov constants between the original system and its normal form.
This leads to the following question.
\begin{description}
  \item[Q1.] Are the Lyapunov constants of the original system~\eqref{ge} algebraically equivalent to its normal form coefficient series?
\end{description}
In this paper, we fill this gap and give a positive answer to {\bf Q1}.
Thanks to algebraic equivalence, we compute the normal form coefficient series, which only involve computations of some linear algebraic systems, instead of Lyapunov constants.
Moreover, we design algorithms for computing normal form coefficient series.
With our algorithm, the normal form coefficient series can be output directly based on
the expression of system~\eqref{ge}.

Based on the algebraic equivalence described in {\bf Q1}, we develop the normal form method to analyze the following question, which is important for both theoretical research and practical applications (see, e.g., \cite{EM2,LIB,LIb,WU}).
\begin{description}
  \item[Q2.] Can two UNSTABLE foci of the subsystems produce a STABLE monodromic singular point of the piecewise-smooth system?
   Moreover, how does its order depend on the orders of the foci of the subsystems?
\end{description}
In this paper, we provide a complete answer to {\bf Q2}.
Our results generalize \cite[Proposition 7]{EM2} from a specific system to general ones, from FF type to FF and FP types, and from order $1$ to arbitrary orders.
Inspired by these results, we apply the theoretical and algorithmic results of this paper to a basic problem in control theory, i.e., the stabilization of switched systems,
and obtain a necessary and sufficient condition for a class of switching signals that guarantees asymptotic stability of a switched nonlinear system.

This paper is organized as follows.
In Section 2, we introduce Lyapunov constants and the normal forms established in \cite{LIJ}.
We present our main results in Section 3, and give the algorithms for computing normal form coefficients in Section 4.
In Section 5, we combine our main results with these algorithms to study the stabilization of switched nonlinear systems.

\section{Lyapunov constants and normal forms}
\setcounter{equation}{0}
\setcounter{lm}{0}
\setcounter{thm}{0}
\setcounter{rmk}{0}
\setcounter{df}{0}
\setcounter{cor}{0}

In this section, we first introduce the classic Poincar\'e-Lyapunov method for the center-focus problem
of the monodromic singular point $O$.
Then for the piecewise-smooth monodromic system~\eqref{ge},
we present the normal form method developed in \cite{LIJ} and make a preparation for proving the algebraic equivalence
between the Lyapunov constants in the classic method and the normal form coefficient series in the normal form method.

Suppose that the nearby orbits of the monodromic singular point $O$ rotate counterclockwise without loss of generality.
Then there exists a sufficiently small real number $\varepsilon> 0$ such that for any $x\in (0, \varepsilon)$ (resp. $x\in (-\varepsilon, 0)$),
the orbit starting from $(x,0)$ will intersect the negative (resp. positive) part of the $x$-axis,
and we denote the first intersection point by $(\Pi^+(x),0)$ (resp. $(\Pi^-(x),0)$).
Define $\Pi^\pm(0)=0$.
The maps
\begin{equation*}
\begin{aligned}
  \Pi^+: \left[0,\varepsilon\right)&\to \mathbb{R}^- \\
  x&\mapsto \Pi^+(x)
\end{aligned}
~~~~{\rm and}~~~~
\begin{aligned}
  \Pi^-: \left(-\varepsilon,0\right]&\to \mathbb{R}^+ \\
  x&\mapsto \Pi^-(x)
\end{aligned}
\end{equation*}
are called the {\it upper} and {\it lower return maps} of the monodromic singular point $O$, respectively.
Following the displacement function method of smooth systems (see \cite{GUC,KUZ,ZZ}),
for system~\eqref{ge} there is a natural way to define a displacement function as
\begin{equation}\label{displd}
\mathcal{D}(x):=\Pi^-\circ\Pi^+(x)-x
\end{equation}
for $x\in \left[0,\varepsilon\right)$.
As shown in Figure \ref{DF}(a),
we can observe that the positive zeros of the displacement function $\mathcal{D}(x)$ correspond to closed orbits of system~\eqref{ge}.
\begin{figure}[htp]
\centering{}
\subfigure[ $\Pi^-(\Pi^+(x))-x$]{
\scalebox{0.66}[0.66]{
\includegraphics{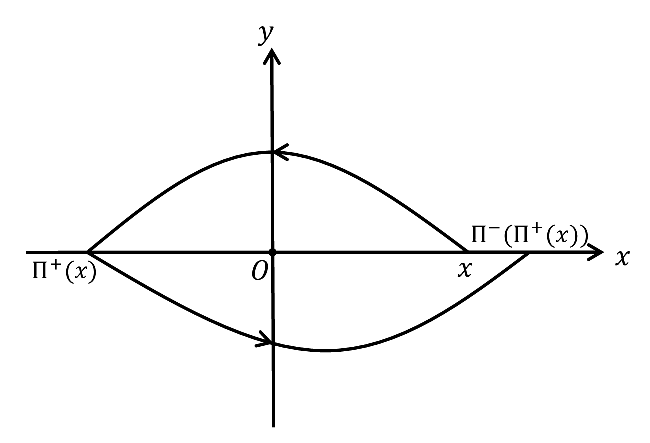}
}
}
\subfigure[$(\Pi^-)^{-1}(x)-\Pi^+(x)$]{
\scalebox{0.66}[0.66]{
\includegraphics{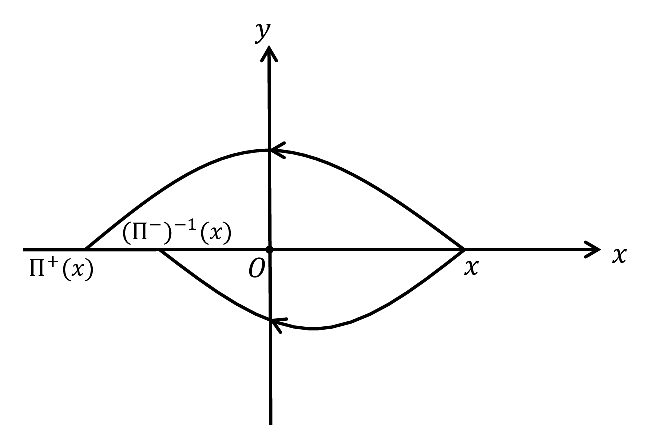}
}
}
\caption{Two equivalent definitions of the displacement function. }
\label{DF}
\end{figure}
The displacement function $\mathcal{D}(x)$ can be expanded at $x=0$ as a power series,
\begin{equation*}
\mathcal{D}(x)=\sum_{k\ge 1}\bar{V}_kx^k,
\end{equation*}
where the coefficient $\bar{V}_k$ is called the {\it $k$-th Lyapunov constant}.

Note that Lyapunov constants are essential for studying the local dynamics of the monodromic singular point $O$ of system~\eqref{ge}.
In fact, on the one hand,
it can be verified that $O$ is a center if and only if $\bar{V}_k=0$ for all $k\ge1$,
and $O$ is a focus if and only if there exists $k\in \mathbb{Z}^+$ such that $\bar{V}_1=\cdots=\bar{V}_{k-1}=0$ and $\bar{V}_k\ne0$.
On the other hand, the sign of the first nonzero Lyapunov constant $\bar{V}_k$ determines the stability of the focus $O$.
Specifically, if $\bar{V}_k<0$ (resp. $\bar{V}_k>0$), then $\mathcal{D}(x)<0$ (resp. $\mathcal{D}(x)>0$) holds for sufficiently small $x$, which implies that $O$ is a stable (resp. unstable) focus.
Clearly, the composition of the half return maps in the displacement function $\mathcal{D}(x)$ defined by \eqref{displd} poses significant difficulties in computing Lyapunov constants.
To avoid the composition, in \cite{CX2,CO,HAN},
another way is to define the difference function by
\begin{equation}\label{DeltaF}
\Delta(x):=(\Pi^-)^{-1}(x)-\Pi^+(x)
\end{equation}
for $x\in \left[0,\varepsilon\right)$
as shown in Figure \ref{DF}(b).
The difference function $\Delta(x)$ can be expanded as
\begin{equation*}
  \Delta(x)=\sum_{k\ge1} V_kx^k
\end{equation*}
at $x=0$.
It is proved in \cite{CX2} that
the first nonzero quantities of series $\{V_k\}_{k=1}^{+\infty}$ and $\{\bar{V}_k\}_{k=1}^{+\infty}$ share the same subscript and differ by a positive constant multiple.
This relationship of those two series is called {\it algebraic equivalence},
denoted by $V_k\sim\bar{V}_k$.
Note that two algebraically equivalent series of Lyapunov constants are equivalent in the study of the center-focus problem and Hopf cyclicity problem.
Thus, in this paper, we consider the difference function $\Delta(x)$ defined by \eqref{DeltaF}, and also call its coefficient $V_k$ the $k$-th Lyapunov constant.

Similar to classical smooth systems,
the index $k$ of the first nonzero Lyapunov constant $V_k$ can be used to determine the order of the monodromic singular point $O$.
Clearly, smooth systems can be regarded as a special case of piecewise-smooth systems with identical upper and lower subsystems.
Therefore, it is necessary to unify the concept of the order of monodromic singular points for both piecewise-smooth and smooth cases, and to address this problem we introduce the fractional order as in \cite{CX1}.
That is,
the monodromic singular point $O$ of system~\eqref{ge} is said to be of order $\varsigma\in\bigcup_{i=0}^{+\infty}\{i/2\}$ if
$V_{2\varsigma+1}\neq0$ and $V_i=0$ for all $i\in\{1,2,...,2\varsigma\}$.
Note that the monodromic singular point $O$ of order $0$ (resp. order $+\infty$) corresponds to a rough focus (resp. a center).
It is worth mentioning that for the piecewise-smooth system~\eqref{ge}, the maximum number of crossing limit cycles that can arise from $O$ via degenerate Hopf bifurcation
is closely related to the order $\varsigma$ of $O$ (see more details in \cite{CX2}).

It is clear that the Poincar\'e-Lyapunov method plays a crucial role in the analysis of the local dynamics and bifurcations of the monodromic singular point $O$.
However, the computation of Lyapunov constants, which employs polar coordinates or general polar coordinates,
involves cumbersome integrals of trigonometric functions, making it challenging even with computer assistance.
To address this, the normal form method proposed in \cite{LIJ} is a good choice.
For any given positive integer $N$,
one can piecewise construct a coordinate transformation
\begin{equation}\label{geni}
  \left(
    \begin{array}{c}
      x \\
      y \\
    \end{array}
  \right)\to
  \left\{\begin{aligned}
  & \left(\begin{array}{cc}
      1 & q_2^+ \\
      0 & q_1^+ \\
    \end{array}\right)
  \left(
    \begin{array}{c}
      x \\
      y \\
    \end{array}
  \right)+
  \left(\begin{array}{c}
      \Phi^+(x,y) \\
      \Psi^+(x,y)\\
    \end{array}\right),~~&&{\rm if}~y\ge0,\\
    &\left(\begin{array}{cc}
      1 & q_2^- \\
      0 & q_1^- \\
    \end{array}\right)
    \left(
    \begin{array}{c}
      x \\
      y \\
    \end{array}
  \right)+
  \left(\begin{array}{c}
       \Phi^-(x,y) \\
      \Psi^-(x,y)\\
    \end{array}\right),&&{\rm if}~y\le0
  \end{aligned}
  \right.
\end{equation}
with $q^\pm_1>0$, $\Phi^+(x,0)=\Phi^-(x,0)$, $\Psi^\pm(x,0)=0$,
and a time scaling
\begin{equation}\label{getsc}
  \mathrm{d}t\to\left\{\begin{aligned}
  &K^+(x,y)\mathrm{d}t,~~&&{\rm if}~y\ge0, \\
  &K^-(x,y)\mathrm{d}t,&&{\rm if}~y\le0 \\
  \end{aligned}\right.
\end{equation}
with $K^\pm(x,y)>0$,
which carry the piecewise-smooth monodromic systems~\eqref{ge} into the FF normal form
  \begin{align}\label{NFF}
		\left( \begin{array}{c}
		\dot{x}\\
		\dot{y}\\
	\end{array} \right) =\left\{
    \begin{aligned}
&		\left(\begin{array}{cc}
       \gamma_1^+ & -1 \\
       1 & \gamma_1^+ \\
     \end{array}\right)
     \left(\begin{array}{c}
         x \\
         y \\
       \end{array}\right)+
       \sum_{i=1}^{N}\gamma_{i+1}^+y^i\left(\begin{array}{c}
         x \\
         y \\
       \end{array}\right)+{\boldsymbol G}^+_{N+2}(x,y),~~~~&&\mathrm{if}~y>0,\\
&		\left(\begin{array}{cc}
       \gamma_1^- & -1 \\
       1 & \gamma_1^- \\
     \end{array}\right)
     \left(\begin{array}{c}
         x \\
         y \\
       \end{array}\right)+
       \sum_{i=1}^{N}\gamma_{i+1}^-y^i\left(\begin{array}{c}
         x \\
         y \\
       \end{array}\right)+{\boldsymbol G}^-_{N+2}(x,y),   && \mathrm{if}~y<0
	\end{aligned} \right.
\end{align}
if $O$ is of FF type,
the FP normal form
      \begin{align}\label{NFP}
		\left( \begin{array}{c}
		\dot{x}\\
		\dot{y}\\
	\end{array} \right) =\left\{
    \begin{aligned}
&		\left(\begin{array}{cc}
       \gamma_1^+ & -1 \\
       1 & \gamma_1^+ \\
     \end{array}\right)
     \left(\begin{array}{c}
         x \\
         y \\
       \end{array}\right)+
       \sum_{i=1}^{N}\gamma_{i+1}^+y^i\left(\begin{array}{c}
         x \\
         y \\
       \end{array}\right)+{\boldsymbol G}^+_{N+2}(x,y),~~~~&&\mathrm{if}~y>0,\\
&	   \left(\begin{array}{c}
         1 \\
         x^{2\ell^--1}\\
       \end{array}\right)+
       {\boldsymbol R}^-_{N}(x,y),  && \mathrm{if}~y<0
	\end{aligned} \right.
\end{align}
if $O$ is of FP type, and
the PP normal form
      \begin{align}\label{NPP}
		\left( \begin{array}{c}
		\dot{x}\\
		\dot{y}\\
	\end{array} \right) =\left\{
    \begin{aligned}
&		\left(\begin{array}{c}
         -1 \\
         x^{2\ell^+-1}+\sum\limits_{i=1}^{N}\sigma_{i+1}^+x^{2\ell^+-1+i}\\
       \end{array}\right)+
      {\boldsymbol R}^+_{N}(x,y),
        && \mathrm{if}~y>0,\\
&		      \left(\begin{array}{c}
         1 \\
         x^{2\ell^--1}\\
       \end{array}\right)+
       {\boldsymbol R}^-_{N}(x,y), && \mathrm{if}~y<0
	\end{aligned} \right.
\end{align}
if $O$ is of PP type, in a topologically equivalent manner.
Here ${\boldsymbol G}^{\pm}_{N+2}(x,y)$ are sums of homogeneous  polynomials of degree greater than or equal to $N+2$,
 ${\boldsymbol R}^{\pm}_{N}(x,y)$ are sums of quasi-homogeneous polynomials of type $(1,2\ell^{\pm})$ with
degree greater than or equal to $N$, and
the multiplicity of the tangent point for the upper (resp. lower) system is set to $2\ell^+-1$ (resp. $2\ell^--1$), where $\ell^\pm\in \mathbb{Z}^+$.

Series
$$\left\{\gamma_k^++(-1)^{k-1}\gamma_k^-\right\}_{k=1}^{+\infty}, ~~~~\left\{\gamma_k^+\right\}_{k=1}^{+\infty},~~~~
\left\{(1+(-1)^k)\sigma_k^+\right\}_{k=1}^{+\infty}$$
are called the {\it normal form coefficient series} of the FF normal form~\eqref{NFF}, FP normal form~\eqref{NFP} and PP normal form~ \eqref{NPP}, respectively.
It is proved in \cite{LIJ} that the Lyapunov constant series (denoted by $\{V^\ast_k\}_{k=1}^{+\infty}$) of the above piecewise-smooth monodromic normal forms are algebraically equivalent to their normal form coefficient series.
More precisely,
\begin{equation}\label{Lyanff}
V_k^\ast\sim \gamma_k^++(-1)^{k-1}\gamma_k^-
\end{equation}
for the FF normal form~\eqref{NFF},
\begin{equation}\label{Lyanfp}
V_k^\ast\sim\gamma_k^+
\end{equation}
for the FP normal form~\eqref{NFP}
and
\begin{equation}\label{Lyanpp}
V_k^\ast\sim \left(1+(-1)^k\right)\sigma_k^+
\end{equation}
for the PP normal form~\eqref{NPP}.

We emphasize that \cite{LIJ} establishes the existence of the piecewise coordinate transformation \eqref{geni} and the time scaling \eqref{getsc},
as well as the topological equivalence between \eqref{ge} and the corresponding normal forms.
What is missing there is the algebraic equivalence of the Lyapunov constant series before and after these
transformations.
Once such an equivalence is established between the Lyapunov constants of \eqref{ge} and those of its normal form, denoted by $V_k \sim V_k^\ast$,
the computation of the Lyapunov constants reduces to that of the normal form coefficient series, which involves only solving explicit linear algebraic systems (see Section~4 for details).
This motivates us to establish $V_k \sim V_k^\ast$ in the next section.


\section{Main results}
\setcounter{equation}{0}
\setcounter{lm}{0}
\setcounter{thm}{0}
\setcounter{rmk}{0}
\setcounter{df}{0}
\setcounter{cor}{0}

In this section, we prove the algebraic equivalence between
the Lyapunov constants of the original piecewise-smooth system~\eqref{ge} and the coefficient series of its normal form.
Based on this algebraic equivalence,
we find the relationship between the order of focus of the original
system and the orders of foci of its subsystems as well as stability.

We first introduce
the following {\it ordinary Bell polynomials} (see \cite{COM})
\begin{equation*}
  \hat{B}_{k,i}\left(v_1,v_2,...,v_{k-i+1}\right)=
  \sum_{S_{k-i+1}}\frac{i!}{b_1!b_2!\cdots b_{k-i+1}!}\prod_{j=1}^{k-i+1}v_j^{b_j},
  \end{equation*}
where
$S_{k-i+1}$ is the set of all $(k-i+1)$-tuples of nonnegative integers $(b_1,...,b_{k-i+1})$ satisfying
$$\sum_{j=1}^{k-i+1}jb_j=k,~~~~ \sum_{j=1}^{k-i+1}b_j=i.$$
Note that the ordinary Bell polynomials play an important role in computing the product of infinite series,
and we have
\begin{equation*}
  {\left(\sum_{j=1}^{\infty}v_jx^j\right)}^n=\sum_{i=n}^{\infty}
  \hat{B}_{i,n}(v_1,...,v_{i-n+1})x^i.
\end{equation*}
Then we state the main results in the following two theorems.

 \begin{thm}\label{dinfulvk}
 For the piecewise-smooth monodromic system~\eqref{ge}, the Lyapunov constant series $\{V_k\}_{k=1}^{+\infty}$ is
 algebraically equivalent to the normal form coefficient series $\{\gamma_k^++(-1)^{k-1}\gamma_k^-\}_{k=1}^{+\infty}$
 (resp. $\{\gamma_k^+\}_{k=1}^{+\infty}$, $\{(1+(-1)^k)\sigma_k^+\}_{k=1}^{+\infty}$)
 if the origin of system~\eqref{ge} is of FF (resp. FP, PP) type, i.e., the first nonzero quantities of these two series share the same subscript and
differ by a positive constant multiple.
\end{thm}

\begin{proof}
With \eqref{Lyanff}, \eqref{Lyanfp} and \eqref{Lyanpp}, we only need to show that the Lyapunov constants before and after the piecewise coordinate transformation~\eqref{geni} and piecewise time scaling~\eqref{getsc} are algebraically equivalent.

We first consider the influence of the piecewise time scaling~\eqref{getsc} on the Lyapunov constants.
Since $K^\pm(x,y)>0$, it is not hard to verify that
\eqref{getsc} does not change the orbital configuration of the two subsystems according to \cite{KUZ}.
Thus, it preserves the two half return maps of the monodromic singular point $O$.
Consequently, it further preserves the difference function $\Delta(x)$, and hence the Lyapunov constants.

The effect of the piecewise coordinate transformation~\eqref{geni} on the first nonzero Lyapunov constant is analyzed in the following three steps.

{\bf Step 1}. Show that the half return maps before and after \eqref{geni} are conjugate.

For the monodromic singular point $O$ of the piecewise-smooth system~\eqref{ge}, we expand the upper half return map $\Pi^+(x)$ and the inverse of the lower half-return map $(\Pi^-)^{-1}(x)$ at $x=0$ as power series
\begin{equation*}
\Pi^+(x)=\sum_{k=1}^{+\infty}\Gamma_k^+ x^k~~\text{and}~
(\Pi^-)^{-1}(x)=\sum_{k=1}^{+\infty}\Gamma_k^- x^k,
\end{equation*}
respectively.
Thus, from \eqref{DeltaF} we get the difference function
\begin{equation*}
\Delta(x)=(\Pi^-)^{-1}(x)-\Pi^+(x)=\sum_{k=1}^{+\infty}
\left(\Gamma_k^--\Gamma_k^+\right)x^k,
\end{equation*}
and the Lyapunov constants $V_k=\Gamma_k^--\Gamma_k^+$ ($k\ge1$).

Since $q_1^\pm>0$, the coordinate transformation \eqref{geni} maps the upper (resp. lower) subsystem of system~\eqref{ge} to the upper (resp. lower) subsystem of the new one.
Hence, we can similarly define the upper half return map (denoted by $\hat{\Pi}^+(x)$) and the lower half return map (denoted by $\hat{\Pi}^-(x)$),
associated with the monodromic singular point $O$ of the new system.
By expanding $\hat{\Pi}^+(x)$ and $(\hat{\Pi}^-)^{-1}(x)$ at $x=0$ as power series
\begin{equation*}
 \hat{\Pi}^+(x)=\sum_{k=1}^{+\infty}\hat{\Gamma}_k^+ x^k~{\rm and}~
(\hat{\Pi}^-)^{-1}(x)=\sum_{k=1}^{+\infty}\hat{\Gamma}_k^- x^k
\end{equation*}
respectively, we obtain that the difference function becomes
\begin{equation*}
\Delta^\ast(x)=( \hat{\Pi}^-)^{-1}(x)- \hat{\Pi}^+(x)=\sum_{k=1}^{+\infty}
    \left(\hat{\Gamma}_k^--\hat{\Gamma}_k^+\right)x^k.
\end{equation*}
Thus, the $k$-th Lyapunov constant (denoted by $V^\ast_k$) is equal to   $\hat{\Gamma}_k^--\hat{\Gamma}_k^+$ for $k\ge1$.

Denote $\mathcal{Q}(x):=x+\Phi^+(x,0)$.
From the form of the transformation~\eqref{geni},
it is not hard to obtain the following commutative diagrams
\begin{equation*}
\begin{CD}
x @>\hat{\Pi}^+>> \hat{\Pi}^+(x)\\
@V\mathcal{Q}VV @VV\mathcal{Q}V\\
\mathcal{Q}(x) @>>\Pi^+> \mathcal{Q}(\hat{\Pi}^+(x))
\end{CD}
\qquad\text{and}\qquad
\begin{CD}
x @>(\hat{\Pi}^-)^{-1}>> (\hat{\Pi}^-)^{-1}(x)\\
@V\mathcal{Q}VV @VV\mathcal{Q}V\\
\mathcal{Q}(x) @>>(\Pi^-)^{-1}> \mathcal{Q}((\hat{\Pi}^-)^{-1}(x)).
\end{CD}
\end{equation*}
Thus, we get that
\begin{equation}\label{Gcoefpo}
\hat{\Pi}^+(x)=\mathcal{Q}^{-1}\circ \Pi^+\circ \mathcal{Q}(x)
\end{equation}
and
\begin{equation}\label{Gcoefne}
(\hat{\Pi}^-)^{-1}(x)=\mathcal{Q}^{-1}\circ (\Pi^-)^{-1}\circ \mathcal{Q}(x).
\end{equation}
This indicates that the half return maps of the piecewise monodromic system~\eqref{ge} before and after the piecewise coordinate transformation \eqref{geni} are conjugate.
Thus, we complete Step 1.

{\bf Step 2}: Derive the relation between $\Gamma_k^\pm$ and $\hat{\Gamma}_k^\pm$.

Since the relation for $\Gamma_k^-$ and $\hat{\Gamma}_k^-$ can be obtained analogously,
we only analyze the relation between $\Gamma_k^+$ and $\hat{\Gamma}_k^+$ in the following.
Note that from \cite{LIJ}, $\mathcal{Q}(x)$ can be expanded at $x=0$ as a formal power series
\begin{equation*}
\mathcal{Q}(x)=x+\sum_{k=2}^{+\infty}\Upsilon_kx^k.
\end{equation*}
From \eqref{Gcoefpo}, it follows that
\begin{equation}\label{Ggohm}
\mathcal{Q}\circ \hat{\Pi}^+(x)=\Pi^+\circ \mathcal{Q}(x).
\end{equation}
For the left-hand side of \eqref{Ggohm}, a straightforward computation yields
\begin{align*}
\mathcal{Q}\circ \hat{\Pi}^+(x)
  &=\sum_{k=1}^{+\infty}\hat{\Gamma}_k^+ x^k+\sum_{k=2}^{+\infty}
  \Upsilon_k\left(\sum_{j=1}^{+\infty}\hat{\Gamma}_j^+ x^j\right)^k\\
  &=\sum_{k=1}^{+\infty}\hat{\Gamma}_k^+x^k
  +\sum_{k=2}^{+\infty}\Upsilon_k\sum_{i=k}^{+\infty}
  \hat{B}_{i,k}\left(\hat{\Gamma}_1^+,...,\hat{\Gamma}_{i-k+1}^+\right)x^i\\
  &=\sum_{k=1}^{+\infty}\hat{\Gamma}_k^+x^k+\sum_{k=2}^{+\infty}
  \sum_{i=2}^{k}\Upsilon_i
  \hat{B}_{k,i}\left(\hat{\Gamma}_1^+,...,\hat{\Gamma}_{k-i+1}^+\right)x^k\\
  &=\hat{\Gamma}_1^+x+\sum_{k=2}^{+\infty}\left\{\hat{\Gamma}_k^+
  +\sum_{i=2}^{k}\Upsilon_i\hat{B}_{k,i}
  \left(\hat{\Gamma}_1^+,...,\hat{\Gamma}_{k-i+1}^+\right)\right\}x^k.
\end{align*}
By setting $\Upsilon_1:=1$,
for the right-hand side of \eqref{Ggohm} it is not hard to derive
\begin{align*}
\Pi^+\circ \mathcal{Q}(x)&=\sum_{k=1}^{+\infty}\Gamma_k^+
    \left(\sum_{j=1}^{+\infty}\Upsilon_jx^j\right)^k \\
    &=\sum_{k=1}^{+\infty}\Gamma_k^+\sum_{i=k}^{+\infty}
    \hat{B}_{i,k}\left(\Upsilon_1,...,\Upsilon_{i-k+1}\right)x^i\\
    &=\sum_{k=1}^{+\infty}\sum_{i=1}^{k}\Gamma_i^+\hat{B}_{k,i}
    \left(\Upsilon_1,...,\Upsilon_{k-i+1}\right)x^k\\
    &=\Gamma_1^+x+\sum_{k=2}^{+\infty}\sum_{i=1}^{k}\Gamma_i^+\hat{B}_{k,i}
    \left(\Upsilon_1,...,\Upsilon_{k-i+1}\right)x^k.
\end{align*}
Thus, by comparing the coefficients of $x^k$ on both sides of \eqref{Ggohm}, we obtain $\hat{\Gamma}_1^+=\Gamma_1^+$ and
\begin{equation}\label{Hatgampo}
\begin{aligned}
    \hat{\Gamma}_k^+
  +\sum_{i=2}^{k}\Upsilon_i\hat{B}_{k,i}
  \left(\hat{\Gamma}_1^+,...,\hat{\Gamma}_{k-i+1}^+\right)
  &=\sum_{i=1}^{k}\Gamma_i^+\hat{B}_{k,i}
    \left(\Upsilon_1,...,\Upsilon_{k-i+1}\right)\\
  &=\Gamma_k^++\sum_{i=1}^{k-1}\Gamma_i^+\hat{B}_{k,i}
    \left(\Upsilon_1,...,\Upsilon_{k-i+1}\right)
  \end{aligned}
\end{equation}
for $k\ge2$.
Similarly, from \eqref{Gcoefne} we can conclude $\hat{\Gamma}_1^-=\Gamma_1^-$ and
\begin{equation}\label{Hatgamne}
\hat{\Gamma}_k^-
  +\sum_{i=2}^{k}\Upsilon_i\hat{B}_{k,i}
  \left(\hat{\Gamma}_1^-,...,\hat{\Gamma}_{k-i+1}^-\right)
  =\Gamma_k^-+\sum_{i=1}^{k-1}\Gamma_i^-\hat{B}_{k,i}
    \left(\Upsilon_1,...,\Upsilon_{k-i+1}\right)
\end{equation}
for $k\ge2$.
This completes Step 2.

{\bf Step 3}: Prove $V_k\sim V^\ast_k$.

Since $V_k=\Gamma_k^--\Gamma_k^+$ and $V^\ast_k=\hat{\Gamma}_k^--\hat{\Gamma}_k^+$ for $k\ge1$,
subtracting \eqref{Hatgampo} from \eqref{Hatgamne} yields
\begin{equation}\label{Rshibhv}
V^\ast_k+\sum_{i=2}^{k}\Upsilon_i\Delta\hat{B}_{k,i}=
  V_k+\sum_{i=1}^{k-1}V_i
  \hat{B}_{k,i}\left(\Upsilon_1,...,\Upsilon_{k-i+1}\right),~~\forall k\ge2,
\end{equation}
where
\begin{equation*}
 \Delta\hat{B}_{k,i}:=\hat{B}_{k,i}
  \left(\hat{\Gamma}_1^-,...,\hat{\Gamma}_{k-i+1}^-\right)-\hat{B}_{k,i}
  \left(\hat{\Gamma}_1^+,...,\hat{\Gamma}_{k-i+1}^+\right).
\end{equation*}
We claim that for any $k\ge1$, if $V_1=\cdots=V_k=0$, then $V^\ast_1=\cdots=V^\ast_k=0$.
In fact,
suppose $m=1$ and $V_1=0$,
then we get
\begin{equation*}
 V^\ast_1=\hat{\Gamma}_1^--\hat{\Gamma}_1^+=\Gamma_1^--\Gamma_1^+=V_1=0.
\end{equation*}
Assume the claim holds for $k=m-1$.
That is,
if $V_1=\cdots=V_{m-1}=0$
then $V^\ast_1=\cdots=V^\ast_{m-1}=0$,
which implies $\hat{\Gamma}_j^-=\hat{\Gamma}_j^+$ for all $j\in\{1,...,m-1\}$.
Consequently,
\begin{equation*}
 \Delta\hat{B}_{m,i}=\hat{B}_{m,i}
  \left(\hat{\Gamma}_1^-,...,\hat{\Gamma}_{m-i+1}^-\right)-\hat{B}_{m,i}
  \left(\hat{\Gamma}_1^+,...,\hat{\Gamma}_{m-i+1}^+\right)=0,~~\forall i\in\{2,...,m\}.
\end{equation*}
Combining this with \eqref{Rshibhv} and $V_1=\cdots=V_{m-1}=0$,
we obtain
\begin{equation}\label{VMBki}
 V^\ast_m=V_m+\sum_{i=1}^{m-1}V_i
  \hat{B}_{m,i}\left(\Upsilon_1,...,\Upsilon_{m-i+1}\right)=V_m,
\end{equation}
and hence, if $V_m=0$ then $V^\ast_m=0$.
By mathematical induction, our claim is proved.
For the first nonzero Lyapunov constant $V_k$ of the monodromic singular point $O$ of system \eqref{ge},
we get $V^\ast_1=\cdots=V^\ast_{k-1}=0$ by our claim and $V^\ast_k=V_k$
by \eqref{VMBki}.
This shows that the coordinate transformation~\eqref{geni} does not change the first nonzero Lyapunov constant.
Therefore, $V_k\sim V^\ast_k$ and
Step 3 is finished.
The proof is completed.
\end{proof}

As indicated in \cite{LIJ}, the coefficients $\Upsilon_k$ ($k\ge2$) in  $\mathcal{Q}(x)$ play an important role in the reduction process,
because the sub-maps individually constructed to reduce the subsystems of \eqref{ge}
are glued together to form a homeomorphism via these coefficients.
This guarantees that the normal forms are obtained in the sense of topological equivalence.
For this reason, we call the coefficients $\Upsilon_k$ ($k\ge2$) the {\it gluing coefficients}.
Theorem~\ref{dinfulvk} further implies that the gluing coefficients are important in proving that the transformation~\eqref{geni} does not change the first nonzero Lyapunov constant.

On the other hand,
Theorem~\ref{dinfulvk} indicates that the normal form method allows us to compute the normal form series instead of Lyapunov constants.
Thereby, the cumbersome integrals of trigonometric functions can be avoided,
and a general algorithm for the normal form coefficient series of the piecewise-smooth monodromic system~\eqref{ge} will be introduced in Section 4.

Based on Theorem~\ref{dinfulvk}, we develop the normal form method to study the relationship between the stability of focus of the original system and the stabilities of foci of its subsystems, as stated in the following theorem.

\begin{thm}\label{Dpsbst}
   For the piecewise-smooth monodromic system~\eqref{ge},
   suppose that the origin $O$ is a stable (resp. an unstable) focus of order $\varsigma^+\in\mathbb{N}$ for the upper subsystem.
   The following statements hold.
\begin{description}
\item[(a)] If $O$ is a stable (resp. an unstable) focus of order $\varsigma^-\in\mathbb{N}$ for the lower subsystem, then
    $O$ is a focus of system~\eqref{ge} of order $\varsigma\in
    \bigcup_{i=1}^{\mathcal{M}}\left\{\frac{2i-1}{2}\right\}\cup\{\mathcal{M}\}$, and it
is stable (resp. unstable) when $\varsigma=\mathcal{M}$, where $\mathcal{M}:=\min\{\varsigma^+,\varsigma^-\}$.
\item[(b)] If $O$ is an invisible tangent point of the lower subsystem, then $O$ is
a focus of system~\eqref{ge} of order $\varsigma\in \bigcup_{i=1}^{\varsigma^+}\left\{\frac{2i-1}{2}\right\}\cup\{\varsigma^+\}$, and it
is stable (resp. unstable) when $\varsigma=\varsigma^+$.
\end{description}
\end{thm}

\begin{proof}
For conclusion {\bf (a)},
we prove the case where the origin $O$ is a stable focus of both subsystems.
The other case can be proved similarly and is omitted here.
By choosing a positive integer $N>2\mathcal{M}+1$,
one can reduce the piecewise-smooth system~\eqref{ge} to the FF normal form~\eqref{NFF}.
Since $O$ is a stable focus of order $\varsigma^\pm$ for the corresponding subsystem,
it follows from \cite[Corollary 4.1]{LIJ} that the coefficients $\gamma_{2i+1}^\pm$ ($i=0,...,\varsigma^\pm$) in the FF normal form~\eqref{NFF} satisfy $\gamma_{2i+1}^\pm=0$ for all $i\in\{0,...,\varsigma^\pm-1\}$ and $\gamma_{2\varsigma^\pm+1}^\pm<0$.

If
$\gamma_{2i}^+=\gamma_{2i}^-$ for all $i\in\{1,...,\mathcal{M}\}$,
combining \eqref{Lyanff} with Theorem~\ref{dinfulvk}, for system~\eqref{ge}
we obtain that the first nonzero Lyapunov constant of $O$ has subscript $2\mathcal{M}+1$, and satisfies
\begin{equation*}
\operatorname{sgn}\left(V_{2\mathcal{M}+1}\right)=\left\{\begin{aligned}
&\operatorname{sgn}\left(\gamma_ {2\varsigma^++1}^+\right),~&&\text{if } \varsigma^+<\varsigma^-,\\
&\operatorname{sgn}\left(\gamma_ {2\varsigma^++1}^++\gamma_ {2\varsigma^-+1}^-\right),&&\text{if } \varsigma^+=\varsigma^-,\\
&\operatorname{sgn}\left(\gamma_{2\varsigma^-+1}^-\right),
&&\text{if } \varsigma^+>\varsigma^-.
\end{aligned}\right.
\end{equation*}
This implies that the origin $O$ is a stable focus of order $\mathcal{M}$ for system \eqref{ge} because $V_{2\mathcal{M}+1}<0$ due to $\gamma_{2\varsigma^\pm+1}^\pm<0$.
If there exists $i\in\{1,...,\mathcal{M}\}$ such that $\gamma_{2j}^+=\gamma_{2j}^-$ for all $j\in\{1,...,i-1\}$ and $\gamma_ {2i}^+\ne\gamma_ {2i}^-$,
then the first nonzero Lyapunov constant of $O$ has subscript $2i$ and satisfies
\begin{equation*}
 \operatorname{sgn}\left( V_{2i}\right)=\operatorname{sgn} \left(\gamma_ {2i}^+-\gamma_ {2i}^-\right)\ne0
\end{equation*}
by \eqref{Lyanff} and Theorem~\ref{dinfulvk}.
Consequently, when $\gamma_ {2i}^+<\gamma_ {2i}^-$ (resp. $\gamma_ {2i}^+>\gamma_ {2i}^-$),
the origin $O$ is a stable (resp. an unstable) focus of order $(2i-1)/2$ for system~\eqref{ge}.
Therefore, the proof of conclusion {\bf (a)} is completed.

For conclusion {\bf (b)},
by choosing a positive integer $N>2\varsigma^++1$
we reduce the piecewise-smooth system~\eqref{ge} into the FP normal form \eqref{NFP}.
Since the origin $O$ is a focus of order $\varsigma^+$ for the upper subsystem,
the coefficients $\gamma_{2i+1}^+$ ($i=0,...,\varsigma^+$) in the FP normal form~\eqref{NFP} satisfy $\gamma_{2i+1}^+=0$ for all $i\in\{0,...,\varsigma^+-1\}$ and $\gamma_{2\varsigma^++1}^+\ne0$
by \cite[Corollary 4.1]{LIJ}.

If $\gamma_{2i}^+=0$ for all $i\in\{1,...,\varsigma^+\}$,
then we obtain that the first nonzero Lyapunov constant of $O$ for system~\eqref{ge} has subscript $2\varsigma^++1$, and satisfies
\begin{equation*}
 \operatorname{sgn}\left( V_{2\varsigma^++1}\right)=\operatorname{sgn}\left( \gamma^+_{2\varsigma^++1}\right)
\end{equation*}
from \eqref{Lyanfp} and Theorem~\ref{dinfulvk}.
This indicates that the origin $O$ is a stable (resp. an unstable) focus of order $\varsigma^+$ provided that
$\gamma_{2\varsigma^++1}^+<0$ (resp. $\gamma_{2\varsigma^++1}^+>0$),
i.e., $O$ is a stable (resp. an unstable) focus of order $\varsigma^+$ for the upper subsystem.
If there exists $i\in\{1,...,\varsigma^+\}$ such that $\gamma_{2j}^+=0$ for all $j\in\{1,...,i-1\}$ and $\gamma_ {2i}^+\ne0$,
then the first nonzero Lyapunov constant of $O$ has subscript $2i$ and satisfies
\begin{equation*}
  \operatorname{sgn}\left(V_{2i}\right)=\operatorname{sgn}\left(\gamma_ {2i}^+\right)
\end{equation*}
according to \eqref{Lyanfp} and Theorem~\ref{dinfulvk}.
Consequently, when $\gamma_ {2i}^+<0$ (resp. $\gamma_ {2i}^+>0$),
the origin $O$ is a stable (resp. an unstable) focus of order $(2i-1)/2$ for system~\eqref{ge}.
Thus, we prove conclusion {\bf (b)} and finish the proof of this theorem.
\end{proof}

Note that in Theorem~\ref{Dpsbst},
we ignore the case where $O$ is a PP type monodromic singular point of system~\eqref{ge},
because it follows from \eqref{Lyanpp} that there is no relationship between the multiplicities of the tangent points $O$ of the subsystems
and the order as well as the stability of $O$ for \eqref{ge}.
In fact, one can easily construct examples where the order $\varsigma$ is greater than, less than, or equal to the multiplicity of the tangent point.

Recently, a specific piecewise-smooth system has been found in \cite{EM2} that takes the origin $O$ as an unstable focus of order $1/2$, while its subsystems take $O$ as a stable focus of order $1$.
It is worth mentioning that Theorem~\ref{Dpsbst} generalizes this result from a specific system to general ones, from FF type to FF and FP types, and from order $1$ to arbitrary orders.
In fact, according to the proof of Theorem~\ref{Dpsbst},
it is not hard to find that for FF type,
two stable (resp. unstable) foci of arbitrary orders can form an unstable (resp. stable) focus of the whole system with any fractional order less than the minimum of their orders;
for FP type,
a stable (resp. unstable) focus of arbitrary order together with a tangent point can form an unstable (resp. stable) focus of the whole system with any fractional order less than the order of the focus.
This result implies that the normal form method can be used to design stabilizing switching signals for switched systems with all unstable subsystems (see Section 5).

\section{Algorithms for computing normal form coefficient series}
\setcounter{equation}{0}
\setcounter{lm}{0}
\setcounter{thm}{0}
\setcounter{rmk}{0}
\setcounter{df}{0}
\setcounter{cor}{0}

In this section, based on the main result Theorem~\ref{dinfulvk} and \cite{LIJ} we provide the algorithms for computing normal form coefficient series, 
which avoids cumbersome integrals of trigonometric functions in the computation of Lyapunov constants.
To simplify the notation, in this section we omit the superscripts $+$ and $-$ used to distinguish the upper and lower subsystems of the piecewise-smooth monodromic system~\eqref{ge} in the reduction process of subsystems.

\subsection{Algorithm for FF type}

In this subsection, we introduce the algorithms for computing FF normal form coefficient series in the following three steps.
For the piecewise-smooth system~\eqref{ge}, suppose that the origin $O$ is a monodromic singular point of FF type.

{\bf Step 1}. Compute the normal form coefficients $\gamma_k^+$.

According to~\cite{CX},
the linear part of the upper subsystem can be changed into
\begin{equation*}
  A:=\left(
    \begin{array}{cc}
      \alpha & -\beta \\
      \beta & \alpha \\
    \end{array}
  \right)
\end{equation*}
by a linear transformation
\begin{equation*}
 L_f: \left(\begin{array}{c}
      x \\
      y \\
    \end{array}
  \right)\to
   \left(
     \begin{array}{cc}
       1 & \frac{1}{2\beta}\frac{\partial X(0,0)}{\partial x}
  -\frac{1}{2\beta}\frac{\partial Y(0,0)}{\partial y} \\
       0 & \frac{1}{\beta}\frac{\partial Y(0,0)}{\partial x} \\
     \end{array}
   \right)
   \left(\begin{array}{c}
     x \\
     y \\
   \end{array}
   \right),
\end{equation*}
where $\alpha\pm \beta i$ are eigenvalues of the Jacobian matrix.
Then the first FF normal form coefficient $\gamma_1^+$ equals $\alpha/\beta$.
Clearly, the higher-order terms after this transformation can be easily computed.

From \cite[Lemma 2.2]{LIJ}, we can find a linear transformation $L_f$ and a series of near-identity transformations
\begin{equation}\label{NisSff}
\left(\begin{array}{c}
x \\
y \\
\end{array}
\right)\to
\left(
\begin{array}{c}
x \\
y \\
\end{array}
\right)+
\left(\begin{array}{c}
\sum\limits_{j=0}^{k}p_{k-j,j}x^{k-j} y^{j} \\
\sum\limits_{j=1}^{k}q_{k-j,j}x^{k-j} y^{j} \\
\end{array}\right),~~k=2,...,m
\end{equation}
that rewrite the upper subsystem of \eqref{ge} as
\begin{equation}\label{mp1for}
\begin{aligned}
\left(\begin{array}{c}
\dot{x} \\
\dot{y} \\
\end{array}\right)=&A
\left(\begin{array}{c}
x \\
y \\
\end{array}\right)
+\sum_{k=1}^{m-1}y^k
\left(\begin{array}{cc}
\nu_{k+1} & -\eta_{k+1} \\
\eta_{k+1} & \nu_{k+1} \\
\end{array}\right)\left(\begin{array}{c}
x \\
y \\
\end{array}\right)\\
&+
\left(\begin{array}{c}
\sum\limits_{i+j= m+1} \hat{a}_{ij} x^iy^j \\
\sum\limits_{i+j= m+1} \hat{b}_{ij} x^iy^j \\
\end{array}\right)
+o\left(|(x,y)|^{m+1}\right).
\end{aligned}
\end{equation}
To compute the normal form coefficient $\gamma_{m+1}^+$,
our task is to find a near-identity transformation
\begin{equation}\label{Hmnio}
\left(\begin{array}{c}
x \\
y \\
\end{array}
\right)\to
\left(
\begin{array}{c}
x \\
y \\
\end{array}
\right)+
\left(\begin{array}{c}
\sum\limits_{j=0}^{m+1}p_{m+1-j,j}x^{m+1-j} y^{j} \\
\sum\limits_{j=1}^{m+1}q_{m+1-j,j}x^{m+1-j} y^{j}\\
\end{array}\right)
\end{equation}
carrying the $(m+1)$-th order terms of system~\eqref{mp1for} into
\begin{equation*}
y^{m}\left(\begin{array}{cc}
\nu_{m+1} & -\eta_{m+1} \\
\eta_{m+1} & \nu_{m+1} \\
\end{array}\right)
\left(\begin{array}{c}
x \\
y \\
\end{array}\right),
\end{equation*}
and thereby the coefficient $\gamma_{m+1}^+$ can be computed by the explicit formula given in \cite{LIJ}.
Associated with this, the following algorithm is designed to compute the normal form coefficient $\gamma_{m+1}^+$ based on the proof of \cite[Lemma 2.2]{LIJ}.
\begin{algorithm}
\caption{Computation of the normal form coefficient $\gamma_{m+1}$}
\begin{algorithmic}[1]
\Require$C$~and~$\hat{a}_{m+1-j,j}, \hat{b}_{m+1-j,j}$ ($j=0,...,m+1$)
\Comment{$C$ denotes the value assigned to the gluing coefficient}
\State $p_{m+1,0}\gets C$, $q_{m+1,0}\gets 0$
\State $p_{m,1}\gets\frac{\hat{a}_{m+1,0}}{\beta}-\frac{m\alpha C}{ \beta}$,
$q_{m,1}\gets\frac{\hat{b}_{m+1,0}}{\beta}+C$
  \For{$k=1,...,m-1$}
\State    $p_{m-k,k+1}\gets \frac{\hat{a}_{m-k+1,k}+(m-k+2)\beta p_{m-k+2,k-1}-m\alpha p_{m-k+1,k}-\beta q_{m-k+1,k}}{(k+1)\beta}$
\State    $q_{m-k,k+1}\gets\frac{\hat{b}_{m-k+1,k}+{\rm sgn}(k-1)(m-k+\!2)\beta q_{m-k+2,k-1}-m\alpha q_{m-k+1,k}+\beta p_{m-k+1,k}}{(k+1)\beta}$
\EndFor
\State $\Upsilon_1\gets\hat{a}_{1,m}-\hat{b}_{0,m+1}+2\beta \left(p_{2,m-1}-q_{1,m}\right)-m\alpha p_{1,m}$
\State  $\Upsilon_2\gets\hat{a}_{0,m+1}+\hat{b}_{1,m}+2\beta \left(p_{1,m}+q_{2,m-1}\right)-m\alpha q_{1,m}$
\State    $p_{0,m+1}\gets\frac{(m+2)\beta \Upsilon_1+m\alpha \Upsilon_2}{m^2\alpha^2+(m+2)^2\beta^2}$,
\State   $q_{0,m+1}\gets\frac{(m+2)\beta \Upsilon_2-m\alpha \Upsilon_1}{m^2\alpha^2+(m+2)^2\beta^2}$
\State $\nu_{m+1} \gets-m\alpha q_{0,m+1}+\beta p_{0,m+1} +\beta q_{1,m}+\hat{b}_{0,m+1}$
\State $\eta_{m+1} \gets m\alpha p_{0,m+1}+\beta q_{0,m+1} -\beta p_{1,m}-\hat{a}_{0,m+1}$
\State   $T_m\gets\frac{1}{\beta}\left(\eta_{m+1} -\sum_{i=1}^{m-1}T_{m-i}\eta_{i+1}\right)$
\State $\gamma_{m+1}\gets\frac{1}{\beta}\left(\nu_{m+1}-
    \sum_{i=1}^{m-1}T_{m-i}\nu_{i+1}-\alpha T_m\right)$
\Ensure $p_{m+1-j,j},q_{m+1-j,j}$ ($j=1,...,m+1$)~and~$\gamma_{m+1}$
\end{algorithmic}
\end{algorithm}
The constant $C$ and the coefficients $\hat{a}_{m+1-j,j}, \hat{b}_{m+1-j,j}$ ($j=0,...,m+1$) of system~\eqref{mp1for} are required in {\bf Algorithm 1}.
Here, the constant $C$ denotes the value assigned to the gluing coefficient $p_{m+1,0}$, which will be chosen in the last step.
The coefficients $\hat{a}_{m+1-j,j}, \hat{b}_{m+1-j,j}$ can be computed when the linear transformation $L_f$ and near-identity transformations~\eqref{NisSff} are known.
Note that besides the normal form coefficient, {\bf Algorithm 1} also outputs the coefficients of the transformation~\eqref{Hmnio},
and hence, we can subsequently compute the coefficients $\hat{a}_{m+2-j,j}$, $\hat{b}_{m+2-j,j}$ ($j=0,...,m+2$).
Therefore, the FF normal form coefficients $\gamma_k^+$ for $k\ge2$ can be obtained by repeating the execution of {\bf Algorithm 1}.

Let ULF$(m,C)$ be the procedure of {\bf Algorithm 1} in a computer algebra system,
where the parameter $m$ means that the normal form coefficient $\gamma_{m+1}^+$ will be computed,
and the parameter $C$ corresponds to the gluing coefficient $p_{m+1,0}$.
In addition, 
let UCF$(m,N)$ be the procedure in a computer algebra system,
where the parameter $m$ indicates that the coefficients $\hat{a}_{m+1-j,j}, \hat{b}_{m+1-j,j}$ ($j=0,...,m+1$) will be computed,
and the parameter $N$ indicates that the normal form coefficients are considered up to order $N+1$.

{\bf Step 2}. Compute the normal form coefficients $\gamma_k^-$.

In \cite{LIJ}, the lower subsystem of system~\eqref{ge} is first mapped to the upper half plane $y>0$,
then simplified using \cite[Lemma 2.2]{LIJ},
and finally mapped back to the lower half plane.
However, since both the coordinate transformation~\eqref{Hmnio} and the lower subsystem can be regarded as defined in a small neighborhood of the origin $O$,
we can directly compute the normal form coefficient $\gamma_{m+1}^-$ using {\bf Algorithm 1}.
That is, the process of computation of the normal form coefficients $\gamma_k^-$ ($k\ge1$) is totally similar to that of $\gamma_k^+$.
More precisely, $\gamma_1^-$ can be computed by the formula given in \cite{CX}, and $\gamma_k^-$ ($k\ge2$) can be computed by repeating {\bf Algorithm 1}.
Let LLF$(m,C)$ and LCF$(m,N)$ be procedures in a computer algebra system, which are used to compute $\gamma_{m+1}^-$ and the coefficients required in {\bf Algorithm 1}, respectively.

{\bf Step 3}. Compute the normal form coefficient series $\gamma^+_k+(-1)^{k-1}\gamma_k^-$.

According to the proof of \cite[Theorem 1.1]{LIJ}, the gluing coefficients for reducing system~\eqref{ge} to its FF normal form can be arbitrarily chosen.
Thus, we choose them to be all zero for simplifying the calculation,
and give {\bf Algorithm 2} to compute normal form coefficient series $\gamma^+_k+(-1)^{k-1}\gamma_k^-$ up to any given number of terms.
\begin{algorithm}
\caption{Computation of the normal form coefficient series $\{\gamma_k^++(-1)^{k-1}\gamma_k^-\}_{k=1}^{N+2}$ }
\begin{algorithmic}[1]
\Require system and $N$
\Comment{$N$ indicates the desire to compute the first $N+2$ terms}
\State compute $\gamma_1^\pm$ and the new higher-order terms
\State LLF($1$, $0$)
\State ULF($1$, $0$)
\For{$r=2,...,N+1$}
\State LCF($r$, $N+1$)
\State LLF($r$, $0$)
\State UCF($r$, $N+1$)
\State ULF($r$, $0$)
\EndFor
\For{$k=1,...,N+2$}
\State $\gamma_k^++(-1)^{k-1}\gamma_k^-$
\EndFor
\Ensure $\gamma_k^++(-1)^{k-1}\gamma_k^-$, $k=1,...,N+2$
\end{algorithmic}
\end{algorithm}

\begin{rmk}\label{nfceosm}
Since any smooth system with a linear focus can be regarded as a piecewise-smooth system~\eqref{ge} with the same subsystems,
it follows that $V_k \sim (1+(-1)^{k-1})\varpi_k$ by Theorem~\ref{dinfulvk} and \eqref{Lyanff},
where $\varpi_k:=\gamma_k^+=\gamma_k^-$.
It follows that a new normal form series $\{\varpi_k(1+(-1)^{k-1})/2\}_{k\ge1}$ is obtained for the smooth systems,
which is algebraically equivalent to the Lyapunov constants series and can be computed by {\bf Algorithm 2}.
Therefore, {\bf Algorithm 2} can be used to study the center-focus problem and the Hopf cyclicity problem for smooth systems.
\end{rmk}

\subsection{Algorithm for FP type}

In this subsection, we introduce the algorithms for computing FP normal form coefficient series in the following two steps.
For the piecewise-smooth system~\eqref{ge}, suppose that the origin $O$ is a monodromic singular point of FP type.

{\bf Step 1}. Compute the gluing coefficients $p_{k,0}$.

The gluing coefficients $p_{k,0}$ ($k\ge2$) are obtained in the reduction of the lower subsystem of \eqref{ge}.
Suppose that the origin $O$ is a tangent point of multiplicity $2\ell-1$.
Then, based on the proof of \cite[Lemma 2.3]{LIJ}, we can construct a linear transformation
\begin{equation*}
L_p:  \left(\begin{array}{c}
      x \\
      y \\
    \end{array}
  \right)\to
   \left(
     \begin{array}{cc}
       1 & 0 \\
       0 & \frac{-1}{(2\ell-1)!a_0}\frac{\partial^{2\ell-1}Y(0,0)}{\partial x^{2\ell-1}} \\
     \end{array}
   \right)
   \left(\begin{array}{c}
     x \\
     y \\
   \end{array}
   \right)
\end{equation*}
and a series of near-identity transformations
\begin{equation}\label{Nisfrfp}
\left(\begin{array}{c}
      x \\
      y \\
    \end{array}
  \right)\to
  \left(
   \begin{array}{c}
      x \\
      y \\
    \end{array}
  \right)+
  \left(
    \begin{array}{c}
     \sum\limits_{i=0}^{\lfloor\frac{k+1}{2\ell}\rfloor}p_{k+1-2\ell i,i}x^{k+1-2\ell i}y^i\\
      \sum\limits_{i=1}^{1+\lfloor\frac{k}{2\ell}\rfloor}q_{k+2\ell-2\ell i,i}x^{k+2\ell-2\ell i}y^i \\
    \end{array}
  \right),~~k=1,...,m
\end{equation}
to rewrite the lower subsystem of \eqref{ge} as
\begin{equation}\label{mp1forP}
\left(\!\!\begin{array}{c}
      \dot{x} \\
      \dot{y} \\
    \end{array}
  \!\!\right)\!\!=\!\!
  \left(\!\!\begin{array}{c}
    a_0-\sum\limits_{k=1}^{m}\mu_{k+1} x^k \\
    -a_0x^{2\ell-1}+\sum\limits_{k=1}^{m}\mu_{k+1} x^{k+2\ell-1} \\
  \end{array}\!\!\right)\!+\!\underbrace{\left(\!\!
    \begin{array}{c}
      \sum\limits_{i=0}^{\lfloor\frac{m+1}{2\ell}\rfloor}\hat{a}_{m+1-2\ell i,i}x^{m+1-2\ell i}y^i\\
      \sum\limits_{i=0}^{1+\lfloor\frac{m}{2\ell}\rfloor}\hat{b}_{m+2\ell-2\ell i,i}x^{m+2\ell-2\ell i}y^i \\
    \end{array}
  \!\!\right)}_{=:\mathcal{G}_m}+h.o.t.,
\end{equation}
where $a_0:=X(0,0)$, $\lfloor \cdot \rfloor$ denotes the floor function,
and $h.o.t.$ represents the higher-order terms in the sense of quasi-homogeneity of type $(1,2\ell)$.
Subsequently, we need to find a near-identity transformation
\begin{equation}\label{Hmniop}
  \left(\begin{array}{c}
      x \\
      y \\
    \end{array}
  \right)\to
  \left(
   \begin{array}{c}
      x \\
      y \\
    \end{array}
  \right)+
  \left(
    \begin{array}{c}
    p_{m+2,0}x^{m+2}+ \sum\limits_{i=1}^{\lfloor\frac{m+2}{2\ell}\rfloor}p_{m+2-2\ell i,i}x^{m+2-2\ell i}y^i\\
      \sum\limits_{i=1}^{1+\lfloor\frac{m+1}{2\ell}\rfloor}q_{m+1+2\ell-2\ell i,i}x^{m+1+2\ell-2\ell i}y^i \\
    \end{array}
  \right)
\end{equation}
which maps the term $\mathcal{G}_m$ of system~\eqref{mp1forP} to $(-\mu_{m+2}x^{m+1},\mu_{m+2}x^{m+2\ell})^\top$.
Accordingly, the transformation~\eqref{Hmniop} can be computed by the following algorithm, and hence the gluing coefficient $p_{m+2,0}$ is obtained.
\begin{algorithm}
\caption{Computation of the gluing coefficient $p_{m+2,0}$}
\begin{algorithmic}[1] 
\Require  $\hat{a}_{i,j}$~and~$\hat{b}_{i,j}$~in~$\mathcal{G}_m$
\State $r\gets\lfloor\frac{m+1}{2\ell}\rfloor$,
$q\gets\lfloor\frac{m+2}{2\ell}\rfloor$
  \If{$r=0$}
\State  $p_{m+2-2\ell,1}\gets0$,
  $q_{m+1,1}\gets\frac{\hat{b}_{m,1}}{a_0(m+1)}$
\ElsIf{$r>0$ and $2\ell \mid (m+2)$}
\State $p_{m+2-2\ell q,q}\gets0$,
$q_{m+1-2\ell r,r+1}\gets\frac{\hat{b}_{m-2\ell r,r+1}}{a_0(m+1-2\ell r)}$
\ElsIf{$r>0$ and $2\ell \mid (m+1)$}
\State $q_{m+1-2\ell r,r+1}\gets 0$
\Else \State $q_{m+1-2\ell r,r+1}\gets\frac{\hat{b}_{m-2\ell r,r+1}}{a_0(m+1-2\ell r)}$
\EndIf
\For{$i=r,r-1,...,1$}
\State
$p_{m+2-2\ell i,i}\gets\frac{\hat{a}_{m+1-2\ell i,i}+{\rm sgn}(q-i)(i+1)a_0p_{m+2-2\ell(i+1),i+1}}{a_0(m+2-2\ell i)}$
\State
$q_{m+1+2\ell-2\ell i,i}\gets\frac{\hat{b}_{m+2\ell-2\ell i,i}+(i+1)a_0q_{m+1-2\ell i,i+1}-(2\ell-1)a_0p_{m+2-2li,i}}{a_0(m+1+2\ell-2\ell i)}$
\EndFor
\State $p_{m+2,0}\gets\frac{\hat{a} _{m+1,0}+\hat{b}_{m+2\ell,0}+a_0(p_{m+2-2\ell,1}+q_{m+1,1})}
 {a_0(m+1+2\ell)}$
 \State
$\mu_{m+2}\gets\hat{b}_{m+2\ell,0}-(2\ell-1)a_0p_{m+2,0}+a_0q_{m+1,1}$
\Ensure $p_{i,j}$~and~$q_{i,j}$ in \eqref{Hmniop}
\end{algorithmic}
\end{algorithm}
Note that in {\bf Algorithm 3},
the coefficients $\hat{a}_{i,j}$ and $\hat{b}_{i,j}$ in $\mathcal{G}_m$ are required,
which can be directly computed once the linear transformation $L_p$ and the near-identity transformations~\eqref{Nisfrfp} are known.
Consequently, by repeating {\bf Algorithm 3}, we get the gluing coefficients $p_{k,0}$ $(k\ge2)$.

Let LLP$(m,\ell)$ be a procedure of {\bf Algorithm 3} in a computer algebra system, where the parameter $m$ indicates that the gluing coefficient $p_{m+2,0}$ will be computed, and
the parameter $\ell$ corresponds to the multiplicity of the tangent point.
Let LCP$(m,N,\ell)$ be a procedure in a computer algebra system,
where the parameters $m$ and $\ell$ have the same meaning as LLP$(m,\ell)$,
and the parameter $N$ indicates that the gluing coefficients $p_{k,0}$ ($k=2,...,N+2$) are desired.

{\bf Step 2}. Compute the normal form coefficient series $\gamma^+_k$.

The normal form coefficients $\gamma^+_k$ ($k\ge1$) are obtained from the reduction of the upper subsystem of \eqref{ge}.
The algorithm used for this reduction is consistent with {\bf Algorithm 1}, but with the gluing coefficients set to the values $p_{k,0}$ ($k=2,...,N+2$) obtained in the previous step.
Based on this, we propose {\bf Algorithm 4} for computing the FP normal form coefficient series $\gamma_k^+$ up to any given number of terms.
\begin{algorithm}
\caption{Computation of the normal form coefficient series $\{\gamma_k^+\}_{k=1}^{N+2}$}
\begin{algorithmic}[1] 
\Require system and $N$
\Comment{$N$ indicates the desire to compute the first $N+2$ terms}
\State compute the new form of \eqref{ge} after the linear transformations $L_f$ and $L_p$
\State LLP(0, $\ell^-$)
\State ULF(1, $p^-_{2,0}$)
\For{$r=1,...,N$}
\State LCP($r-1$, $N-1$, $\ell^-$)
\State LLP($r$, $\ell^-$)
\State UCF($r+1$, $N+1$)
\State ULF($r+1$, $p^-_{r+2,0}$)
\EndFor
\For{$k=1,...,N+2$}
\State $\gamma_k^+$
\EndFor
\Ensure $\gamma^+_k$, $k=1,...,N+2$
\end{algorithmic}
\end{algorithm}

\subsection{Algorithm for PP type}

In this subsection, we introduce the algorithms for computing PP normal form coefficient series in the following two steps.
For the piecewise-smooth system~\eqref{ge}, suppose that the origin $O$ is a monodromic singular point of PP type.

{\bf Step 1}. Compute the gluing coefficients $p_{k,0}$.

As stated in Step 1 of Section 4.2,
the gluing coefficient $p_{k,0}$ can be computed by repeatedly applying {\bf Algorithm 3}.

{\bf Step 2}. Compute the normal form coefficient series $(1+(-1)^k)\sigma_k^+$

The normal form coefficients $\sigma_k^+$ ($k\ge2$) arise from the reduction of the upper subsystem of \eqref{ge}.
In the following, we introduce the process of computation of $\sigma_{m+2}^+$.
According to the proof of \cite[Theorem 1.1]{LIJ}, we can find a linear transformation $L_p$ and a series of near-identity transformations~\eqref{Nisfrfp} that change the upper subsystem of \eqref{ge} into
\begin{equation}\label{NcforPPt}
\left(\!\!\begin{array}{c}
      \dot{x} \\
      \dot{y} \\
    \end{array}
  \!\!\right)\!\!=\!\!
  \left(\!\!\begin{array}{c}
    a_0-\sum\limits_{k=1}^{m}\nu_{k+1} x^k \\
    -a_0x^{2\ell-1}+\sum\limits_{k=1}^{m}\eta_{k+1} x^{k+2\ell-1} \\
  \end{array}\!\!\right)\!+\!\underbrace{\left(\!\!
    \begin{array}{c}
      \sum\limits_{i=0}^{\lfloor\frac{m+1}{2\ell}\rfloor}\hat{a}_{m+1-2\ell i,i}x^{m+1-2\ell i}y^i\\
      \sum\limits_{i=0}^{1+\lfloor\frac{m}{2\ell}\rfloor}\hat{b}_{m+2\ell-2\ell i,i}x^{m+2\ell-2\ell i}y^i \\
    \end{array}
  \!\!\right)}_{=:\mathcal{H}_m}+h.o.t..
\end{equation}
In what follows, we need to construct a near-identity transformation~\eqref{Hmniop} that carries the term $\mathcal{H}_m$ of system~\eqref{NcforPPt} to $(-\nu_{m+2}x^{m+1},\eta_{m+2}x^{m+2\ell})^\top$.
Then the normal form coefficient $\sigma_{m+2}^+$ can be explicitly computed via the coefficients $\nu_k, \eta_k$ ($k=2,...m+2$) of system~\eqref{NcforPPt}.
Accordingly, the following {\bf Algorithm 5} is designed to compute $\sigma_{m+2}^+$.
Note that the coefficients $\hat{a}_{i,j}, \hat{b}_{i,j}$ in $\mathcal{H}_m$ and the gluing coefficient $p_{m+2,0}$ are required in {\bf Algorithm 5}.
In fact, $\hat{a}_{i,j}, \hat{b}_{i,j}$ can be directly computed because the linear transformation $L_p$ and near-identity transformations~\eqref{Nisfrfp} are known,
and $p_{m+2,0}$ can be found according to the first step.
 \begin{algorithm}
\caption{Computation of the normal form coefficient $\sigma_{m+2}^+$}
\begin{algorithmic}[1] 
\Require  $\hat{a}_{i,j}, \hat{b}_{i,j}$~in~$\mathcal{H}_m$~and~$p_{m+2,0}$
\State $r\gets\lfloor\frac{m+1}{2\ell}\rfloor$,
$q\gets\lfloor\frac{m+2}{2\ell}\rfloor$
  \If{$r=0$}
\State  $p_{m+2-2\ell,1}\gets0$,
  $q_{m+1,1}\gets\frac{\hat{b}_{m,1}}{a_0(m+1)}$
\ElsIf{$r>0$ and $2\ell \mid (m+2)$}
\State $p_{m+2-2\ell q,q}\gets0$,
$q_{m+1-2\ell r,r+1}\gets\frac{\hat{b}_{m-2\ell r,r+1}}{a_0(m+1-2\ell r)}$
\ElsIf{$r>0$ and $2\ell \mid (m+1)$}
\State $q_{m+1-2\ell r,r+1}\gets 0$
\Else \State $q_{m+1-2\ell r,r+1}\gets\frac{\hat{b}_{m-2\ell r,r+1}}{a_0(m+1-2\ell r)}$
\EndIf
\For{$i=r,r-1,...,1$}
\State
$p_{m+2-2\ell i,i}\gets\frac{\hat{a}_{m+1-2\ell i,i}+{\rm sgn}(q-i)(i+1)a_0p_{m+2-2\ell(i+1),i+1}}{a_0(m+2-2\ell i)}$
\State
$q_{m+1+2\ell-2\ell i,i}\gets\frac{\hat{b}_{m+2\ell-2\ell i,i}+(i+1)a_0q_{m+1-2\ell i,i+1}-(2\ell-1)a_0p_{m+2-2li,i}}{a_0(m+1+2\ell-2\ell i)}$
\EndFor
\State
$\nu_{m+2}\gets-\hat{a}_{m+1,0}+(m+2)a_0p_{m+2,0}-a_0p_{m+2-2\ell,1}$
 \State
$\eta_{m+2}\gets\hat{b}_{m+2\ell,0}-(2\ell-1)a_0p_{m+2,0}+a_0q_{m+1,1}$
 \State
 $\sigma_{m+2}\gets-\frac{\eta_{m+2}}{a_0}-\sum_{j=1}^{m+1}
\sum_{i=1}^{j}\frac{\hat{B}_{j,i}
  (\nu_2,...,\nu_{j-i+2})\eta_{m+2-j}}{a_0^{i+1}}$
 \Ensure $p_{i,j}, q_{i,j}$ in \eqref{Hmniop}~and~$\sigma_{m+2}$
\end{algorithmic}
\end{algorithm}

Let ULP$(m,\ell)$ be a procedure of {\bf Algorithm 5} in a computer algebra system, where the parameters $m$ and $\ell$ have the same meaning as LLP$(m,\ell)$.
And let UCP$(m,N,\ell)$ be a procedure in a computer algebra system that used to find the coefficients $\hat{a}_{i,j}, \hat{b}_{i,j}$ required in {\bf Algorithm 5},
where the parameters $m$, $N$ and $\ell$ have the same meaning as LCP$(m,N,\ell)$.
The gluing coefficient $p_{m+2,0}$ required in {\bf Algorithm 5} can be computed by procedures LLP$(m,\ell)$ and LCP$(m,N,\ell)$ as stated in the first step.

Based on the previous analysis, we establish {\bf Algorithm 6} for computing the PP normal form coefficients series $(1+(-1)^k)\sigma^+_{k}$ up to any given number of terms.
\begin{algorithm}
\caption{Computation of the normal form coefficient series $\{(1+(-1)^k)\sigma^+_{k}\}_{k=1}^{N+2}$}
\begin{algorithmic}[1] 
\Require system and $N$
\Comment{$N$ indicates the desire to compute the first $N+2$ terms}
\State compute the new form of \eqref{ge} after the linear transformation $L_p$
\State $\eta_1^+\gets -a_0^+$
\State LLP($0$, $\ell^-$)
\State ULP($0$, $\ell^+$)
\State $\sigma^+_2\gets \frac{\nu^+_2-\eta^+_2}{a_0^+}$
\For{$r=1,...,N$}
\State LCP($r-1$, $N-1$, $\ell^-$)
\State LLP($r$, $\ell^-$)
\State UCP($r-1$, $N-1$, $\ell^+$)
\State ULP($r$, $\ell^+$)
\EndFor
\For{$k=1,...,N+2$}
\State $(1+(-1)^k)\sigma^+_{k}$
\EndFor
\Ensure $(1+(-1)^k)\sigma^+_{k}$, $k=1,...,N+2$
\end{algorithmic}
\end{algorithm}

\section{Application to stabilization of switched systems}

\setcounter{equation}{0}
\setcounter{lm}{0}
\setcounter{thm}{0}
\setcounter{rmk}{0}
\setcounter{df}{0}
\setcounter{cor}{0}

Clearly, once the expression of the piecewise-smooth system~\eqref{ge} is known and $N$ is determined,
we can input them into the corresponding algorithm given in Section 4 depending on the type of the monodromic singular point of system~\eqref{ge}.
Then the first $N+2$ terms of the normal form coefficient series will be output,
which can be further used in the study of the center-focus problem and the Hopf cyclicity problem by applying Theorem~\ref{dinfulvk}.

In this section, motivated by Theorem~\ref{Dpsbst},
we consider applying the normal form method together with the algorithms established in Section 4
to analyze an important practical problem, i.e., the switching control problem (see \cite{LIB} for more details).
More precisely, we obtain a necessary and sufficient condition for a class of switching signals such that a switched nonlinear system,
which has a monodromic singular point $O$ of FF type,
is asymptotically stable.

Consider the $n$-th order switched system used in control theory,
which has the form
\begin{equation}\label{Swswithss}
  \dot{\boldsymbol z}={\boldsymbol f}_{\sigma}({\boldsymbol z}),~{\boldsymbol z}\in \mathbb{R}^n,
\end{equation}
where $\sigma$ takes values in $\{1,...,k\}$ and
$\{{\boldsymbol f}_i: i=1,...,k\}$ is a family of smooth functions from $\mathbb{R}^n$ to $\mathbb{R}^n$.
Note that, as in \cite{SUN,ZHA}
the function $\sigma$ is usually called a {\it switching signal} (or {\it switching law}), which may
depend on time $t$ or state ${\boldsymbol z}$ or both.

Assume that each subsystem has the origin $O$ as an equilibrium, i.e., ${\boldsymbol f}_i({\boldsymbol 0})={\boldsymbol 0}$ for $i=1,...,k$.
If $O$ is asymptotically stable (in the sense of Lyapunov stability), then the switched system~\eqref{Swswithss} is termed asymptotically stable.
In \cite{LIb}, three basic problems regarding stability and design of switched system~\eqref{Swswithss} are listed,
and the third one is as follows.

{\bf Problem C} \cite{LIb}. Construct a switching signal that makes the switched system~\eqref{Swswithss} asymptotically stable.

Clearly, if at least one of the subsystems is asymptotically stable, then
the above problem can be solved by taking $\sigma\equiv i$, where $i$ is the subscript of an asymptotically stable subsystem.
Thus, in the context of {\bf Problem C}, it is understood that none of the individual subsystems is asymptotically stable.
The task of solving {\bf Problem C} by designing such a switching signal is commonly known as {\it stabilization of switched systems}.

The stabilization of the second order switched linear system
\begin{equation}\label{SwswL}
(\dot{x}, \dot{y})^\top=A_{\sigma} (x, y)^\top
\end{equation}
has received much attention (see, e.g., \cite{IWA, SUN, WIC, XU}),
where the switching signal $\sigma$ takes values in $\{1,2\}$.
Note that the conical switching signals are usually designed to stabilize system~\eqref{SwswL} (see e.g., \cite{WU, YUA}).
Under such state-dependent signals, the state space is divided into some conical regions by several lines or rays intersecting at the origin.
Subsequently, stabilization is achieved by appropriately selecting subsystems on these conical regions.
In particular, a special case of conical switching signals of the form
\begin{equation}\label{Sss12}
  \sigma=\left\{\begin{aligned}
  &1,~~~&&{\rm if}~y-Kx>0,\\
  &2,&&{\rm if}~y-Kx<0
  \end{aligned}\right.
\end{equation}
is termed a {\it linear switching signal}.
It is not hard to verify that the switched linear system~\eqref{SwswL} consisting of two subsystems each with an unstable focus,
cannot be stabilized by the linear switching signal~\eqref{Sss12}, according to Theorem~\ref{Dpsbst}.

In this paper, we consider the stabilization of the following second order switched nonlinear system
\begin{equation}\label{Swsnlg}
  (\dot{x}, \dot{y})^\top={\boldsymbol f}_{\sigma}(x, y),
\end{equation}
where $\mathcal{D}_1(x,y):=a_1x+a_2y+a_3x^2+a_4xy+a_5y^2$,
$\mathcal{D}_2(x,y):=b_1x+b_2y+b_3x^2+b_4xy+b_5y^2$ and
\begin{equation*}
  {\boldsymbol f}_{i}(x, y)=\left(
  \begin{array}{c}
      -y \\
      x+\mathcal{D}_i(x,y)y \\
    \end{array}
  \right),~~i\in\{1,2\}.
\end{equation*}
Note that each subsystem $(\dot{x}, \dot{y})^\top={\boldsymbol f}_{i}(x, y)$ corresponds to the generalized Li\'enard equation
$\ddot{x}+\mathcal{D}_i(x,\dot{x})\dot{x}+x=0$,
and its term $\mathcal{D}_i(x,y)y$ corresponds to a nonlinear damping force.
Suppose that $(a_1,...,a_5, b_1,...,b_5)\in \mathbb{R}^{10}$, $3a_5+a_3-a_1a_2>0$ and $3b_5+b_3-b_1b_2>0$.
Combining Remark~\ref{nfceosm} and {\bf Algorithm 2},
we obtain that the first nonzero terms of the normal form series of the subsystem $(\dot{x},\dot{y})^\top={\boldsymbol f}_{1}(x, y)$ (resp. $(\dot{x},\dot{y})^\top={\boldsymbol f}_{2}(x, y)$) is
\begin{equation}\label{detstba}
  \varpi_3=\frac{3a_5+a_3-a_1a_2}{4}>0~~~\left({\rm resp.}~ \varpi_3=\frac{3b_5+b_3-b_1b_2}{4}>0\right).
\end{equation}
This implies that the origin $O$ is an unstable focus of order $1$ for the two subsystems.

Inspired by Theorem~\ref{Dpsbst}, we can construct a linear switching signal~\eqref{Sss12} to stabilize system~\eqref{Swsnlg},
and obtain the following theorem.
\begin{thm}\label{DSofas}
  System~\eqref{Swsnlg} with a linear switching signal~\eqref{Sss12} is locally asymptotically stable if and only if the slope $K$ satisfies
  \begin{equation}\label{stsscon}
    (b_1-a_1)K^3+(3a_2-3b_2)K^2+2a_2-2b_2<0.
  \end{equation}
\end{thm}

\begin{proof}
  System~\eqref{Swsnlg} with a linear switching signal~\eqref{Sss12} takes the form
\begin{equation}\label{Swsnlgwits}
\left( \begin{array}{c}
\dot{x}\\
\dot{y}\\
\end{array} \right)=\left\{
\begin{aligned}
&{\boldsymbol f}_{1}(x, y),~~ &&{\rm if}~y-Kx>0,\\
&{\boldsymbol f}_{2}(x, y),~~ &&{\rm if}~y-Kx<0.
\end{aligned} \right.
\end{equation}
It is not hard to verify that the linear transformation $(x,y)^\top\to L(x,y)^\top$ equivalently maps system~\eqref{Swsnlgwits} into the new form
\begin{equation}\label{Swsnlgwitsnew}
\begin{aligned}
\left( \begin{array}{c}
\dot{x}\\
\dot{y}\\
\end{array} \right)&=\left\{
\begin{aligned}
&L^{-1}\cdot{\boldsymbol f}_{1}\circ L(x, y),~~ &&{\rm if}~y>0,\\
&L^{-1}\cdot{\boldsymbol f}_{2}\circ L(x, y),~~ &&{\rm if}~y<0,
\end{aligned} \right.\\
&=\left\{
\begin{aligned}
&
\left(
\begin{array}{c}
-y-Kx \\
(1+K^2)x+Ky+\mathcal{D}_1(x,y+Kx)(y+Kx) \\
\end{array}
\right),~~ &&{\rm if}~y>0,\\
&\left(
\begin{array}{c}
-y-Kx \\
(1+K^2)x+Ky+\mathcal{D}_2(x,y+Kx)(y+Kx) \\
\end{array}
\right),~~ &&{\rm if}~y<0,
\end{aligned} \right.
\end{aligned}
\end{equation}
where
\begin{equation*}
  L:=\left(
    \begin{array}{cc}
     1 & 0 \\
     K & 1 \\
    \end{array}
  \right).
\end{equation*}
Clearly, the origin $O$ is a monodromic singular point of FF type for the new system.
Thus, from {\bf Algorithm 2},
we compute that the first normal form coefficient equals $0$ and the second one equals
\begin{equation}\label{SecnNormal}
  \frac{(b_1-a_1)K^3+(3a_2-3b_2)K^2+2a_2-2b_2}{3K^{2}+3}.
\end{equation}
Denote the numerator of \eqref{SecnNormal} by $\mathcal{N}(K)$.
It follows from Theorem~\ref{dinfulvk} that the origin $O$ is a stable (resp. unstable) focus of the system~\eqref{Swsnlgwitsnew} when $\mathcal{N}(K)<0$ (resp. $\mathcal{N}(K)>0$).
If $\mathcal{N}(K)=0$, from Theorem~\ref{Dpsbst}{\bf (a)} and \eqref{detstba} we can deduce that the origin $O$ is an unstable focus.
Therefore, the proof is finished.
\end{proof}

A deterministic switching signal is said to possess {\it dwell time} $\tau$ if $t_{i+1}-t_i\geq\tau$ for any two consecutive switching times $t_i$ and $t_{i+1}$.
As stated in \cite{SUN}, a good switching signal should satisfy that it not only makes the switched system stable, but also avoids fast switching, preferably with a guaranteed positive dwell time.
There exists a sufficiently small neighborhood of the origin $O$ and a number
$\varepsilon>0$ such that, when the system~\eqref{Swsnlg} is expressed in polar coordinates, $1-\varepsilon<\dot{\theta}<1+\varepsilon$ holds.
And hence, the linear switching signal~\eqref{Sss12} used to stabilize the switched system~\eqref{Swsnlg} has a positive dwell time $\tau>\pi/(1+\varepsilon)>0$.
It is worth noting that the method presented in Theorem~\ref{DSofas} can be used to stabilize more general switched systems with two unstable subsystems that possess a higher-order weak focus.

We employ the following example in order to verify Theorem~\ref{DSofas}.
\begin{exam}\label{exa}
  Stabilize the switched nonlinear system $(\dot{x}, \dot{y})^\top={\boldsymbol f}_{\sigma}(x, y)$ with
  \begin{equation*}
  {\boldsymbol f}_{1}(x, y)=\left(
  \begin{array}{c}
      -y \\
      x+(x+2y+x^2+y^2)y \\
    \end{array}
  \right)~~~{\rm and}~~~
{\boldsymbol f}_{2}(x, y)=\left(
  \begin{array}{c}
      -y \\
      x+(6x+y+2x^2+2y^2)y \\
    \end{array}
  \right).
\end{equation*}
\end{exam}

For the two subsystems, one can verify that the coefficients satisfy \eqref{detstba}.
Thus, the origin $O$ is an unstable focus, and their local phase portraits are shown in Figure~\ref{fili}.
Since both subsystems are unstable, we consider stabilizing the switched system in Example~\ref{exa} by a linear switching signal~\eqref{Sss12}.
According to Theorem~\ref{DSofas}, asymptotic stability is achieved if and only if
\begin{equation*}
5K^3+3K^2+2=(K+1)(5K^2-2K+2)<0,
\end{equation*}
i.e., $K\in (-\infty,-1)$.
In what follows, to verify this result we take $K=-2$ as an example,
and use the switched signal~\eqref{Sss12}$|_{K=-2}$ to stabilize the switching nonlinear system given in Example~\ref{exa},
as shown in Figure~\ref{keo}.
\begin{figure}[h]
\centering{}
\subfigure[Unstable focus $O$ of $(\dot{x},\dot{y})^\top={\boldsymbol f}_{1}(x, y)$]{
	\scalebox{0.3}[0.3]{
	\includegraphics{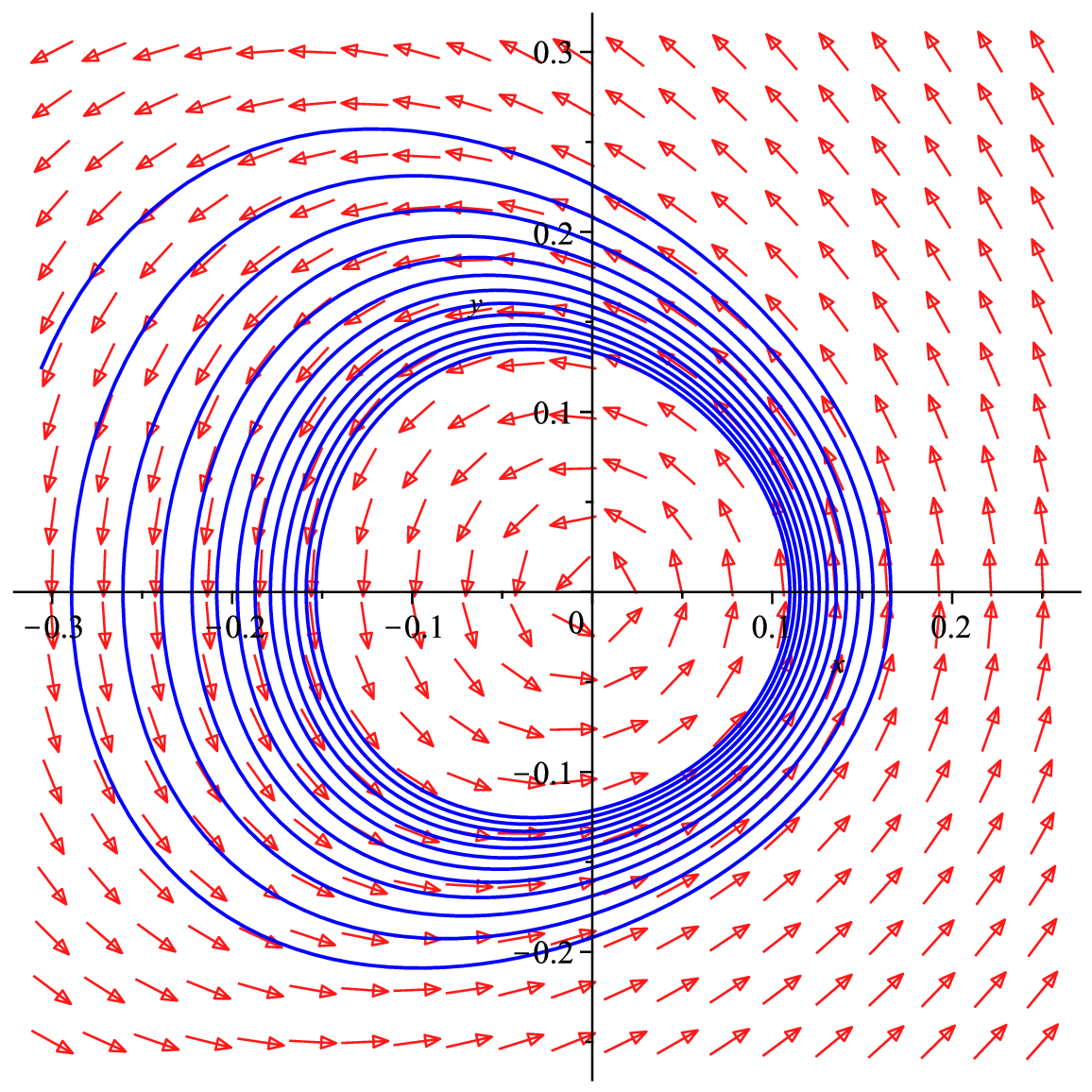}
	}
}~~~~~~
\subfigure[Unstable focus $O$ of $(\dot{x},\dot{y})^\top={\boldsymbol f}_{2}(x, y)$]{
	\scalebox{0.3}[0.3]{
	\includegraphics{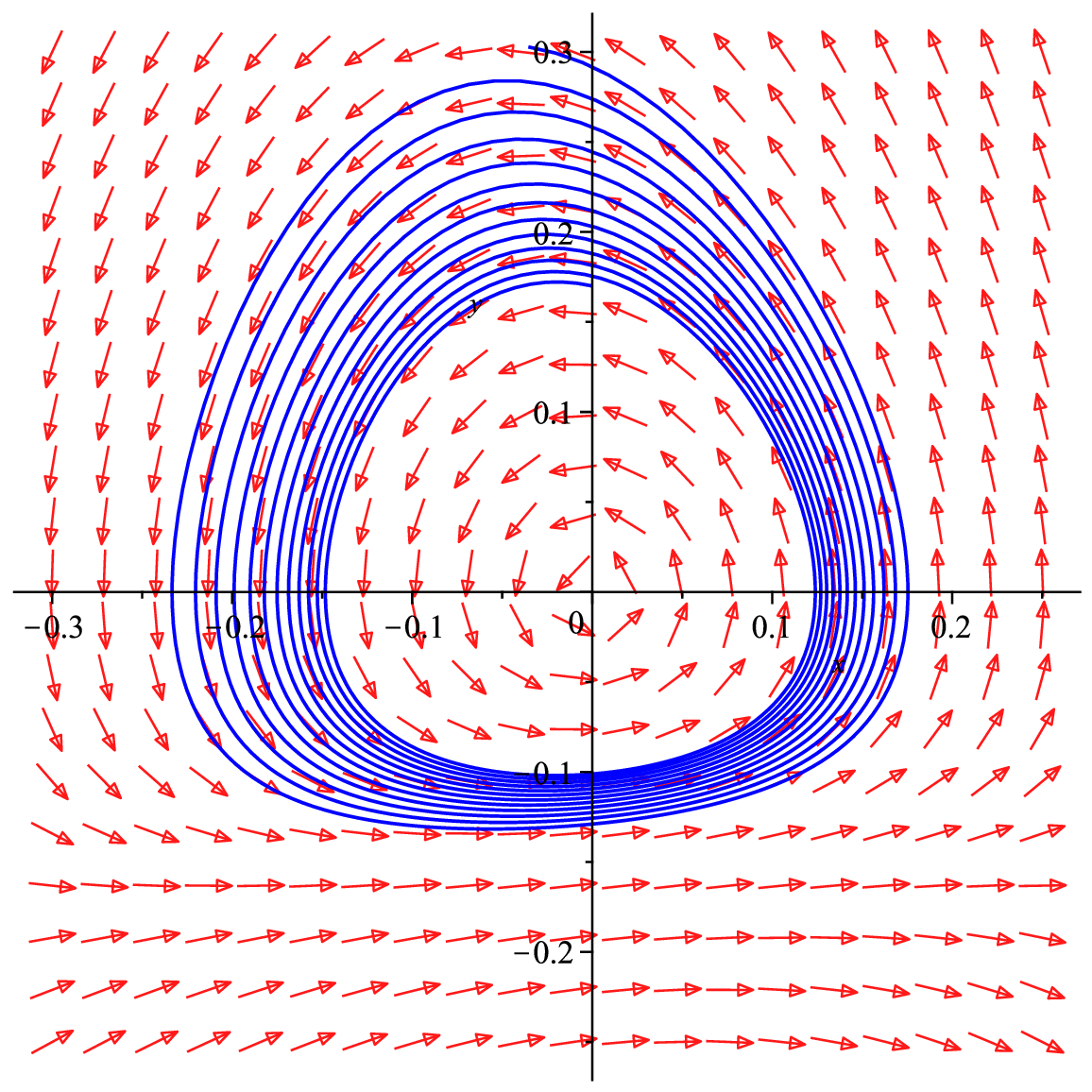}
	}
}
\caption{Local phase portraits of the two subsystems near the origin $O$.}
\label{fili}
\end{figure}
\begin{figure}[h]
  \centering
  \includegraphics[width=6.5cm]{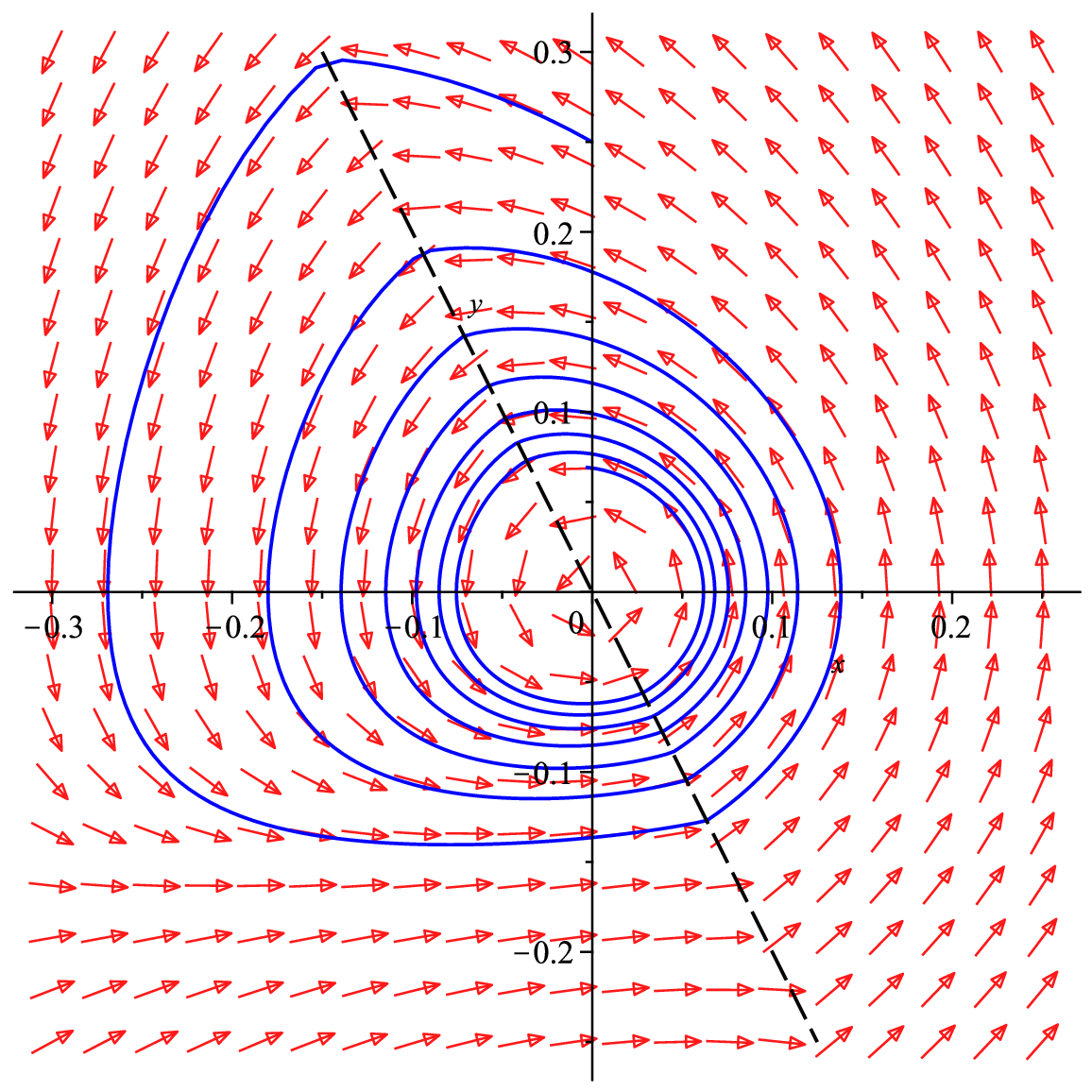}
  \caption{Stabilization of the switched system by the linear switching signal~\eqref{Sss12}$|_{K=-2}$.}
  \label{keo}
\end{figure}

Note that although the parameters in Example~\ref{exa} are chosen as specific numbers for clarity,
the conclusion of Theorem~\ref{DSofas} depends only on the condition~\eqref{stsscon},
which is satisfied for an open set of parameters near the chosen values. Therefore, the stabilizing effect of the linear switching signal~\eqref{Sss12} is robust under small perturbations of the system coefficients, and the same computational procedure applies to any switched system of the form~\eqref{Swsnlg} fulfilling condition~\eqref{detstba}.

\section*{Acknowledgments}

This work is financially supported by the National Key R \& D Program of China (No. 2022YFA1005900), National Science Foundation of China (No. 12271378), Sichuan Science and Technology Program (No. 2024NSFJQ0008) and NSFSPC (No. 2026NSFSC0008).

\section*{Declarations and data availability}

All authors contributed equally to this work and are listed alphabetically. The authors declare no competing interests.
No data was used for the research described in the article.

\end{document}